\documentclass{cimart}

\usepackage{amsrefs}
\usepackage{float}
\usepackage{mathrsfs}
\usepackage{tikz}
\usepackage{tikz-cd}
\usetikzlibrary{decorations.markings,arrows.meta}

\newcommand{\2}{^{[2]}}
\newcommand{\abs}[1]{\vert{#1}\vert}
\newcommand{\e}[1]{\overline{#1}}
\newcommand{\F}{\mathbb{F}}
\newcommand{\gen}[1]{\langle{#1}\rangle}
\newcommand{\mc}[1]{\mathscr{#1}}
\newcommand{\mf}[1]{\mathfrak{#1}}
\newcommand{\li}{\mathscr{L}}
\newcommand{\pow}{^{[2]^n}}
\newcommand{\pres}[2]{\left\langle{#1}\, \big\vert\, {#2}\right\rangle}
\newcommand{\set}[2]{\left\lbrace{#1}\ \big\vert\ {#2}\right\rbrace}
\newcommand{\Z}{\mathbb{Z}}

\DeclareMathOperator{\ad}{ad}
\DeclareMathOperator{\cd}{cd}
\DeclareMathOperator{\ext}{Ext}
\DeclareMathOperator{\gr}{gr}
\DeclareMathOperator{\hnn}{HNN}
\DeclareMathOperator{\invlim}{\varprojlim}
\DeclareMathOperator{\spn}{Span}
\DeclareMathOperator{\tor}{Tor}

\title{Restricted graph Lie algebras in characteristic two}

\authors{Simone Blumer}

\authorinfo{University of Vienna, Austria}{simoneblumer96@gmail.com}

\abstract{
    We investigate restricted Lie algebras arising as analogues of (twisted) right-angled Artin groups and right-angled Coxeter groups over fields of characteristic two. These algebras are defined via quadratic relations determined by decorated graphs. We compute their cohomology rings with trivial coefficients and uncover phenomena specific to characteristic two. Unlike in zero or odd characteristics, where quadratically defined ordinary and restricted Lie algebras have equivalent cohomology theories, the characteristic two case exhibits dependence on the base field. In particular, we prove that the ground field being the prime field $\F_2$ characterizes when a Lie-theoretic analogue of the twisted Droms theorem holds.
    }



\acknowledgments{The author would like to thank Julian Feuerpfeil, Thomas Weigel, and Claudio Quadrelli for helpful discussions. He is especially thankful to Dietrich Burde for supporting his postdoctoral position at the University of Vienna, and to Claudio Quadrelli for bringing to life in \LaTeX{} the author’s fish-shaped graph. The author thanks the anonymous referees for their careful reading and their positive and helpful reports. The author has been supported by the Austrian Science Foundation (FWF), Grant DOI 10.55776/P33811. For open access purposes, the author has applied a CC BY public copyright license to any author-accepted manuscript version arising from this submission.}

\VOLUME{35}
\YEAR{2027}
\ISSUE{2}
\NUMBER{1}
\DOI{https://doi.org/10.46298/cm.17062}
\licence{CC BY 4.0}
\editinfo{December 9, 2025}{March 17, 2026}{Ivan Kaygorodov and  Maxime Fairon}
\begin{document}
	{
	
	\newtheorem{thm}{Theorem}[section]
	\newtheorem*{thmA}{Main Theorem}
	\newtheorem*{thmB}{Theorem B}
	\newtheorem*{thm*}{Theorem}
	\newtheorem{cor}[thm]{Corollary}
	\newtheorem{lem}[thm]{Lemma}
	\newtheorem{prop}[thm]{Proposition}
	\newtheorem{defin}[thm]{Definition}
	\theoremstyle{definition}
	\newtheorem{exam}[thm]{Example}
	\theoremstyle{definition}
	\newtheorem{examples}[thm]{Examples}
	\newtheorem{rem}[thm]{Remark}
	\newtheorem{case}{\sl Case}
	
	\newtheorem{claim}{Claim}
	\newtheorem{fact}[thm]{Fact}
	\newtheorem{question}[thm]{Question}
	\newtheorem*{questionss}{Questions}
	\newtheorem{conj}[thm]{Conjecture}
	\newtheorem*{notation}{Notation}
	\swapnumbers
	\newtheorem{rems}[thm]{Remarks}
	
	\theoremstyle{definition}
	\newtheorem*{acknowledgment}{Acknowledgment}

	\numberwithin{equation}{section}
}
\tikzset{
	mid arrow/.style={
		postaction={
			decorate,
			decoration={
				markings,
				mark=at position 0.65 with {\arrow{>}},
			}
		}
	}
}


\section{Introduction}
Restricted Lie algebras, introduced by Jacobson \cite{jacob}, provide a natural framework for studying Lie-theoretic structures in positive characteristic. Their additional $p$-power operation $x \mapsto x^{[p]}$ mimics the action of the Frobenius map on associative algebras of characteristic $p,$ and endows them with richer algebraic properties than ordinary Lie algebras.

Restricted Lie algebras naturally arise as graded objects associated with distinguished filtrations of pro-$p$ groups (see, e.g., \cite{analyticprop}). 
They also play a crucial role in the study of algebraic groups (see \cite{demazure}).

On the other hand, quadratic algebras, namely those defined by relations that are quadratic in the generators, form an important class of algebras which serve as a natural first approximation to the study of more general graded algebras. Among these, Koszul (restricted) Lie algebras are determined by their cohomology ring. They have been studied by Weigel \cite{weig} (see also \cite{sb_kosz}). The primary motivation for this paper stems from current research in Galois theory. 
Specifically, the Norm Residue Isomorphism Theorem by Rost and Voevodsky \cite{voe} (formerly known as the Bloch--Kato conjecture; see Weibel's survey \cite{weibelBK}) implies that the $\F_p$-cohomology rings $H^\bullet(G_\F(p),\F_p)$ of maximal pro-$p$ quotients of absolute Galois groups $G_\F$ are quadratic algebras.
If $k$ is a field of positive characteristic $p,$ and $\li$ is an ordinary Lie algebra with a given presentation over $k,$ then there is an associated restricted Lie algebra $\mf g$ defined by the same presentation, or, equivalently, by taking the primitive elements of the universal enveloping algebra $U(\li)$ (cf. \cites{sb_kosz,milnorMoore}). The assignment $\li\mapsto\mf g$ gives rise to the restrictification functor (or $p$-hull), which provided that $p$ is odd, defines an isomorphism between two categories. The first consists of Lie algebras whose relations are quadratic in the generators, and the second of restricted Lie algebras with the same property. Moreover, the cohomology theories of these objects are the same. This case has already been studied in the author's paper \cite{sb_kosz}. 

However, in characteristic $2$ the situation fundamentally changes, as the ``squares'' of some generators can appear in the quadratic defining relations of restricted Lie algebras. Consequently, the restrictification functor is reduced to a mere embedding, producing strong consequences both on a cohomological side and on the structure of the subalgebras (see \cite{sb_kosz}).
People working on modular algebraic structures in characteristic $2$ ironically say that ``$2$ is the \textit{oddest} prime number''. Despite these complexities, working in characteristic $2$ also opens up the possibility of considering a richer and broader range of examples. 

The first part of the paper is devoted to the study of homological properties of restricted Lie algebras. For example, we compute the Hilbert series of graded restricted Lie algebras under mild assumptions.
\begin{thm}
	Let $\mf g=\mf g_1\oplus\mf g_2\oplus\dots$ be a graded, restricted Lie algebra with finite-dimensional cohomology and eigenvalues $\lambda_1,\dots,\lambda_n.$ Then, for all $m\geq 1,$ \[\dim\mf g_m=\sum_{i=1}^nM_{2,m}(\lambda_i),\]
	where $M_{2,m}$ is the generalized necklace polynomial mod-$2$ of degree $m.$
\end{thm}

Right-angled Artin groups (RAAGs), also known as partially commutative groups, or graph groups, form a well-studied class of combinatorially defined abstract groups whose properties are encoded within simple graphs. 

A distinguished classifying complex of a RAAG can be constructed by attaching tori of various dimensions (see Salvetti \cite{salvetti}).  By additionally including Klein bottles into this construction, one gets a modified version of RAAGs, called twisted right-angled Artin groups, or T-RAAGs. These groups are characterized by defining relations that are either standard commutators $[x,y]=1,$ or Klein-type relations $yxy^{-1}=x^{-1}$ between the generators. Notably, the Klein-type relation can also be expressed as $[x,y]=xyx^{-1}y^{-1}=x^2,$ which allows us to consider the restricted Lie algebra defined by the analogous presentation. These objects are similarly encoded in a graph which also has some directed edges.

Droms \cite{droms} determined the class of graphs whose associated RAAGs are locally RAAGs. Specifically, he showed that the (finitely generated) subgroups of a RAAG are all RAAGs if and only if the defining graph is a Droms graph, that is, it contains neither a square graph nor paths of length $3$ as induced subgraphs.
Prompted by Droms' work, a recent work of Foniqi, Quadrelli, and the author \cite{sb_droms} extends this result to T-RAAGs.

The present paper aims to study the Lie-algebraic counterparts of these groups. Associated to a graph with some directed edges, we define a restricted Lie algebra, called the twisted right-angled Artin graded (or T-RAAG) Lie algebra, that is generated by the vertices of the graph, subjected to a specific defining relation for each plain/directed edge. Under mild assumptions on the defining graph, the associated T-RAAG is a Koszul restricted Lie algebra.

\begin{thm}
	Let $\Gamma$ be a mixed graph, and $\mf a(\Gamma)$ be the associated T-RAAG restricted Lie algebra over a field $k$ of characteristic two. Then, each of the following implies that $\Gamma$ is a Droms mixed graph:
	\begin{enumerate}
		\item All standard subalgebras of $\mf a(\Gamma)$ are T-RAAGs.
		\item All standard subalgebras of $\mf a(\Gamma)$ are quadratic restricted Lie algebras.
		\item The cohomology ring of $\mf a(\Gamma)$ is universally Koszul.
	\end{enumerate}
	Moreover, these conditions are equivalent to the fact that $\Gamma$ is a mixed Droms graph if and only if $k$ is the prime field $\mathbb F_2.$
\end{thm}
Recall from \cite{sb_kosz} that a standard subalgebra of a graded, restricted Lie algebra  $\mf g=\sum_{i\geq 1}\mf g_i$ is a restricted subalgebra generated by elements of degree $1$ of $\mf g$ (see \S \ref{sec:restr} for the definitions).

We also give a complete characterization of those graphs for which all the standard subalgebras of the associated T-RAAG Lie algebra over a field $k\neq \F_2$ are T-RAAG Lie algebras.

\begin{thm}
	Let $k\neq \F_2$ be any field of characteristic $2.$ Then, the following are equivalent: 
	\begin{enumerate}
		\item $\Gamma$ is a Droms mixed graph, and all the directed edges in any connected component of $\Gamma$ have a common origin.
		\item All standard subalgebras of $\mf a(\Gamma)$ are T-RAAGs.
		\item All standard subalgebras of $\mf a(\Gamma)$ are quadratic restricted Lie algebras.
		\item The cohomology ring of $\mf a(\Gamma)$ is universally Koszul.
	\end{enumerate} 
\end{thm}

Analogously, we examine the Lie algebraic version of right-angled Coxeter groups, which are specific quotients of RAAGs where all the canonical generators have order two. To achieve a Droms-type result for these, it becomes necessary to extend our study to a broader class of restricted Lie algebras. These are defined by simple graphs whose vertices are labelled by elements of $\F_2.$ 
The corresponding generators of the restricted Lie algebra have torsion depending on their labelling. For these Lie algebras, the analogue of Droms' theorem proves to be independent of the ground field.
\begin{thm}
	Let $\Gamma=(V,E)$ be a graph endowed with a function $\theta:V\to \F_2=\{0;1\},$ and let $k$ be a field of characteristic $2.$ Let $\mf e(\Gamma,\theta)$ be the associated extended right-angled Lie algebra (ERA) with presentation \[\pres{v:\ v\in V}{\theta(u)u\2,\ [v,w]:\ u,v,w\in V,\ \{v,w\}\in E}.\] Then, the following are equivalent: 
	\begin{enumerate}
		\item $\Gamma$ is a Droms graph that does not contain any path-graph of length $2$ on which $\theta$ is non-constant, and where the middle vertex $v$ has $\theta(v)=0.$
		\item The graph $\Gamma$ is Droms, and, for all connected induced subgraphs $\Lambda$ of $\Gamma,$ either $\theta\vert_{V(\Lambda)}\equiv 0,$ or there exists a central vertex $v$ of $\Lambda$ with $\theta(v)=1.$
		\item All the standard subalgebras of the ERA Lie algebra $\mf e(\Gamma)$ are ERA Lie algebras.
	\end{enumerate}
\end{thm}

We also consider the Bloch-Kato property for restricted Lie algebras. A restricted Lie algebra $\mathfrak{g}$ is said to be weakly Bloch-Kato if all of its standard restricted subalgebras are quadratic, and Bloch-Kato if all of its standard restricted subalgebras are Koszul. Such property is intended to mimic the behaviour of subgroups of maximal pro-$p$ quotients of $G_\F$ for some field $\F.$ Unlike the zero/odd characteristic cases, where (restricted) Lie algebras are Koszul when all of their subalgebras generated in degree $1$ are quadratic (see \cite{sb_kosz}), in characteristic two this is false in general. Nevertheless, we demonstrate that within the classes of restricted Lie algebras under consideration, this property holds true.  
\begin{thm}
	Let $\mf g$ be either a T-RAAG Lie algebra or an ERA Lie algebra over an arbitrary field of characteristic $2.$ Then the following are equivalent: 
	\begin{enumerate}
		\item The restricted Lie algebra $\mf g$ is weakly Bloch-Kato.
		\item The restricted Lie algebra $\mf g$ is Bloch-Kato.
	\end{enumerate}
\end{thm}

\section{Restricted Lie algebras of characteristic two}
If not otherwise stated, henceforth we denote by $k$ an arbitrary field of characteristic two, and by $\F_2$ its prime field.

\subsection{Restricted Lie algebras}\label{sec:restr}
Let $V$ be a vector space over $k.$
The \textbf{free restricted Lie algebra} on $V$ is the (ordinary) Lie subalgebra of the commutator Lie algebra structure of the tensor $k$-algebra $T(V)=k\oplus V\oplus (V\otimes_{k}V)\oplus \dots$ that is generated by the $2^n$-th powers\[v^{2^n}=\underbrace{v\otimes \dots\otimes v}_{2^n\mbox{ \tiny  times}}\]of elements $v\in V,$ for $n\geq 0.$ 
We denote it by $\mf f(V),$ and, for elements $x,y\in \mf f(V),$ we write $[x,y]$ for their commutator $x\otimes y+y\otimes x,$ and $x\2$ for the element $x\otimes x$; we also set $x^{[2]^{n+1}}:=(x\pow)\2.$ It follows from our definition that $[x,x]=0$ for all $x\in \mf f(V).$
Considering $T(V)$ as a Hopf algebra in the usual way, $\mf f(V)$ is also the set of its primitive elements. 

The following identities hold true for all $\lambda\in k$ and $x,y\in \mf f(V)$:
\begin{enumerate}
	\item $(\lambda x)\2=\lambda^2x\2$;
	\item $[x\2,y]=[x,[x,y]]$;
	\item\label{brackets2} $(x+y)\2=x\2+[x,y]+y\2.$
\end{enumerate}

These define the notion of a $2$-\textbf{operation} on an arbitrary Lie algebra $\mf g.$ A Lie algebra over $k$ with a $2$-map is called a \textbf{restricted Lie algebra} (of characteristic two). Notice that, in the light of (\ref{brackets2}), the Lie brackets are determined by the $2$-map. 

Let $\mf g$ be a restricted Lie algebra.
A subalgebra (resp. an ideal) of the ordinary Lie algebra $\mf g$ is called a restricted subalgebra (resp. a restricted ideal) if it is closed under the $2$-map. For a subset $X\subseteq \mf g,$ we denote by $(X)$ the restricted ideal of $\mf g$ generated by $X.$
If $\mf g$ is a $\Z_+$-graded vector space $\mf g=\bigoplus_{i=1}^\infty\mf g_i$ such that $[x,y]\in \mf g_{i+j}$ and $x\2\in \mf g_{2i}$ for $x\in \mf g_i$ and $y\in \mf g_j,$ then we call $\mf g$ a (positively) \textbf{graded restricted Lie algebra}. Clearly, $\mf f(V)$ can be given a natural grading induced by the tensor degree of the tensor algebra. 
For instance, if $V=kx\oplus ky$ has dimension $2,$ then $\mf f(V)_1=V,$ and $\mf f(V)_2=kx\2\oplus k[x,y]\oplus ky\2.$
An ordinary Lie homomorphism $f:\mf g\to \mf h$ between two restricted graded Lie algebras will be simply called a restricted homomorphism if it satisfies $f(\mf g_i)\subseteq \mf h_i,$ and commutes with the $2$-map. The kernel of a restricted homomorphism $\mf g\to\mf h$ is a restricted ideal of $\mf g,$ and the image of a restricted subalgebra of $\mf g$ is a restricted subalgebra of $\mf h.$

If $R$ is a subset of $\mf f(V),$ then the quotient of $\mf f(V)$ by the restricted ideal generated by $R$ is denoted by $\pres{V}{R}$ (or by $\pres{v_i}{r_i}$ if $(v_i)$ and $(r_i)$ are bases of $V$ and $R,$ respectively). If $R$ is a homogeneous subspace of $\mf f(V),$ then $\pres{V}{R}$ is a graded restricted Lie algebra generated in degree $1,$ and, conversely, any graded Lie algebra $\mf g$ generated in degree $1$ is naturally isomorphic with $\pres{\mf g_1}{R},$ where $R$ is the kernel of the natural map $\mf f(\mf g_1)\to \mf g.$ Graded restricted Lie algebras generated in degree $1$ are called \textbf{standard}. For an arbitrary standard restricted Lie algebra $\mf g,$ we call the natural map $\mf f(\mf g_1)\to \mf g$ the \textbf{free cover} of $\mf g$; by abuse of language, $\mf f(\mf g_1)$ is also called the free cover of $\mf g.$ 
If $R=R_2\oplus R_3\oplus\dots$ is the homogeneous decomposition of the kernel of the free cover of $\mf g,$ then the restricted Lie algebra \[\q\mf g:=\pres{\mf g_1}{R_2},\] together with the natural map $\q\mf g\to \mf g,$ is called the \textbf{quadratic cover} of $\mf g.$ We say that $\mf g$ is \textbf{quadratic} if the map $\q\mf g\to \mf g$
is an isomorphism.

The \textbf{restricted} (universal) \textbf{envelope} $u(\mf g)$ of a $2$-restricted Lie algebra $\mf g$ is the quotient of the universal enveloping algebra of the ordinary Lie algebra $\mf g$ by the two-sided ideal generated by the elements $x\2-x^2$ for $x\in \mf g.$ One can realize $u(\mf g)$ as the quotient of the tensor algebra $T(\mf g)$ by the two-sided ideal generated by the following elements
\begin{enumerate}
	\item[(a)] $[x,y]+x\otimes y+y\otimes x,$ where $x,y\in \mf g;$
	\item[(b)] $x\2+x^2,$ where $x\in \mf g.$
\end{enumerate}
By the Milnor-Moore Theorem \cite{milnorMoore}, the algebra $u(\mf g)$ is an augmented Hopf algebra, whose set of primitive elements is $\mf g.$ 

If $\mf g$ is a graded restricted Lie algebra and $M$ is a graded restricted $\mf g$-module, that is $x\cdot m\in M_{i+j},$ and $x\2\cdot m=x\cdot(x\cdot m)$ for $x\in \mf g_i$ and $m\in M_j,$ we define the (bigraded restricted) \textbf{cohomology groups} of $\mf g$ as 
\[H^{ij}(\mf g,M)=\ext^{ij}_{u(\mf g)}(k,M)\] where $k\simeq u(\mf g)/\mf g u(\mf g)$ is the trivial $1$-dimensional module of $u(\mf g)$ concentrated in degree $0$ (see \cites{hochschild_restr,evansfuchs,quentin}). 

For $M=k,$ the whole cohomology $\bigoplus_{i,j\geq 0}H^{ij}(\mf g,k),$ simply denoted $H^\bullet(\mf g),$ can be endowed with a multiplication mapping (called the cup-product), making it a bigraded connected algebra over $k.$ The cohomological dimension $\cd\mf g$ of $\mf g$ is the supremum of the integers $n$ such that $H^n(\mf g)\neq 0$ (cf. \cite{weig}).
Recall that $\mf g$ is said to be a \textbf{Koszul} restricted Lie algebra if $H^{ij}(\mf g)=0$ unless $i=j.$

Henceforth, restricted Lie algebras are all assumed to be finitely generated; for example, for a standard restricted Lie algebra $\mf g,$ we assume $\dim\mf g_1<\infty.$

Recall from the Introduction the following crucial definition.
\begin{defin}
	A standard restricted Lie algebra $\mf g$ is \textbf{Bloch-Kato} (resp. \textbf{weakly Bloch-Kato}) if all of its standard restricted subalgebras are Koszul (resp. quadratic). 
\end{defin} The proof of \cite{sb}*{Thm. A} still applies to the restricted context, implying that a standard restricted Lie algebra $\mf g$ is Bloch-Kato iff its cohomology ring $H^\bullet(\mf g)$ is \textbf{universally Koszul}, which amounts to saying that all the quotients by ideals generated by elements of $H^1(\mf g)$ admit a linear resolution as $H^\bullet(\mf g)$-modules (see \cite{conca} for the original definition in the commutative context).

\begin{thm}
	A standard restricted Lie algebra over a field of positive characteristic is Bloch-Kato if and only if its restricted cohomology ring $H^\bullet(\mf g)$ is a universally Koszul algebra.
\end{thm}

It is worth noting that the restricted cohomology of a restricted Lie algebra $\mf g$ is not the same as the Chevalley-Eilenberg cohomology of the underlying ordinary Lie algebra $\mf g,$ which is the $\ext$-group of its ordinary universal enveloping algebra, i.e., \[H^\bullet_o(\mf g,M):=\ext^\bullet_{U(\mf g)}(k,M).\] However, the epimorphism $U(\mf g)\to u(\mf g)$ induces a natural map $H^\bullet(\mf g,M)\to H^\bullet_o(\mf g,M).$

\begin{exam}\label{ex:1dim}In this example, we consider quadratic restricted Lie algebras generated by a single element.
	
	(1) Let $\mf k$ be the free restricted Lie algebra on one generator. Its cohomology ring is the $2$-dimensional algebra $k[T]/(T^2).$ As an ordinary Lie algebra, $\mf k$ is the abelian Lie algebra of countable dimension, which is thus infinitely generated, i.e., its ordinary first cohomology group is non-zero in each internal degree $j$ that is a $2$-power. The Chevalley-Eilenberg cohomology of $\mf k$ is thus the exterior algebra over the vector space $k^{\mathbb N}$ of all sequences in $k.$ 
	
	The restricted Lie algebra associated to the Zassenhaus filtration of any powerful pro-$p$ group of rank $n$ is the direct product $\mf k^n$ of $n$ copies of $\mf k$ (see \cite[\S 2]{analyticprop}).
	
	(2) Consider the $1$-dimensional restricted Lie algebra $\mf g=\mf k/\mf k\2=\pres{x}{x\2}.$ The (ordinary) enveloping algebra $U(\mf g)$ of the ordinary Lie algebra $\mf g\simeq k$ is isomorphic to the polynomial ring $k[T],$ and hence the ordinary cohomology ring of $\mf g$ is the dual ring $k[\varepsilon]/(\varepsilon^2).$ In particular, it is a finite-dimensional algebra. 
	
	On the other hand, the restricted envelope $u(\mf g)$ is isomorphic to the quotient ring $k[S]/(S^2),$ proving that its cohomology ring is the polynomial ring $k[\delta],$ which is infinite-dimensional. 
	
	The restricted Lie algebra of the finite abelian simple group $C_2$ is clearly isomorphic to $\mf k/\mf k\2.$ 
\end{exam}
We say that an element $x$ of a restricted Lie algebra $\mf g$ is a torsion element if $x\pow=0$ for some $n\geq 1.$ Thus $\mf g$ is \textbf{torsion-free} if the zero vector is the unique torsion element. Torsion-free abelian Lie algebras are called \textbf{free abelian}. On the contrary, $\mf g$ is \textbf{elementary abelian} if $x\2=0$ holds for every $x\in \mf g.$ In fact, notice that this also implies that $[x,y]=0,$ for all $x,y\in \mf g.$ 
For example, the restricted Lie algebra $\mf k$ from Example \ref{ex:1dim}(1) is free abelian, while its quotient $\mf k/\mf k\2$ of Example \ref{ex:1dim}(2) is elementary abelian.

If $V$ is a $k$-vector space, the \textbf{exterior algebra} $\Lambda(V)$ of $V$ is the quotient of the tensor algebra $T(V)=k\oplus V\oplus (V\otimes_k V)\oplus\dots$ by the ideal generated by the elements $x\otimes x,$ for $x\in V.$ Notice that, since $k$ has characteristic two, the ideal properly contains the ideal generated by $x\otimes y+y\otimes x,$ for $x,y\in V,$ which is, instead, the defining ideal of the \textbf{symmetric algebra} $S(V)$ as a quotient of $T(V).$

\begin{exam}\label{ex:nonBKds}
	Consider the direct product $\mf g\times \mf k$ of the restricted Lie algebra  $\mf g=\pres{v,w}{[v,w]+v\2}$ with the free abelian Lie algebra $\mf k=\gen z$ of rank $1,$ and its subalgebra $\mf h=\gen{v+z,w}.$
	
	One has $[v+z,w]=v\2,$ and $(v+z)\2=v\2+z\2,$ which imply that $$\mf h_2=k\cdot v\2+k\cdot z\2+k\cdot w\2$$ has dimension $3.$ In particular, the quadratic cover $\q\mf h$ is a free restricted Lie algebra of rank $2.$ However, $[[v+z],w\2]=[v\2,w]=[v,v\2]=0$ is a non-trivial relation of $\mf h,$ proving that it is not quadratic. 
	
	By the Corollary \ref{cor:TRAAGkosz} below, $\mf g$ is Koszul, and the proper standard subalgebras are all free abelian, hence Koszul. In particular, $\mf g$ is Bloch-Kato but $\mf g\times \mf k$ is not. Notice that $\mf g$ is not the restrictification of any ordinary quadratic Lie algebra.
\end{exam}
In particular, unlike the case of zero/odd characteristic, the direct product of a Bloch-Kato Lie algebra with a free abelian Lie algebra is not always Bloch-Kato. 
\begin{rem}
	Since the cohomology of a direct product of restricted Lie algebras is the wedge product of the cohomology rings of the factors, we see that the twisted extension of a universally Koszul algebra is not always universally Koszul, in contradiction with Proposition 31 of \cite{MPPT}. The flaw in that proposition sits at the end of the proof where the authors define the map $\mathbin{\rotatebox[origin=c]{270}{$\xi$}}:B\to B$ by $a+a'x\mapsto a+a'(x-l)$ for $a,a'\in A,$ claiming that it is an automorphism. Nevertheless, this is not well-defined in general as $\mathbin{\rotatebox[origin=c]{270}{$\xi$}}(x)^2=(x-l)^2=x^2+l^2=tx+l^2$ and $\mathbin{\rotatebox[origin=c]{270}{$\xi$}}(tx)=t(x-l).$ It is well-defined only if $tl=l^2.$
\end{rem}

We now show that the Bloch-Kato property is preserved if the second factor is elementary abelian instead of free abelian.

\begin{lem}\label{lem:directbk}
	Let $\mf g$ be a (weakly) Bloch-Kato restricted Lie algebra. Then $\mf g\times \mf k/\mf k\2$ is (weakly) Bloch-Kato.
\end{lem}
\begin{proof}
	Let $\mf m$ be a standard subalgebra of $\mf g\times \mf k/\mf k\2,$ and fix a non-zero element $z$ of $\mf k/\mf k\2.$ If $\mf m$ is a standard subalgebra of $\mf g,$ then it is Koszul (resp. quadratic) as $\mf g$ is Bloch-Kato (resp. weakly Bloch-Kato). On the other hand, if $\mf m$ is not contained in $\mf g,$ there exist elements $x_0,x_1,\dots,x_n\in \mf c=\mf c(\Gamma)$ such that $\mf m=\gen{x_1,\dots,x_n,x_0+z}.$
	The standard subalgebra $\mf h=\gen{x_0,x_1,\dots,x_n}$ of $\mf g$ is Koszul (resp. quadratic) as $\mf g$ is (resp. weakly) Bloch-Kato. We can suppose that $x_0$ is independent over $k$ from $x_1,\dots,x_n,$ as otherwise $\mf h=\gen{x_1,\dots,x_n},$ whence we would get that $\mf m=\mf h\times \gen{z}$ is Koszul (resp. quadratic). Now, the map \[\phi:\mf h\to\mf m: \quad x_0\mapsto x_0+z,\quad x_i\mapsto x_i\ (i=1,\dots,n)\] is a well-defined isomorphism. Indeed, if $r=\sum_{0\leq i<j\leq n}\alpha_{ij}[x_i,x_j]+\sum_{i=0}^n\beta_i x_i\2$ is a relation in $\mf h,$ then 
	\begin{align*}
		\sum_{0\leq i<j\leq n}\alpha_{ij}&[\phi(x_i),\phi(x_j)]+\sum_{i=0}^n\beta_i \phi(x_i)\2=\\
		&=\sum_{j=1}^n\alpha_{0j}[x_0+z,x_j]+\beta_0(x_0+z)\2+\sum_{1\leq i<j\leq n}\alpha_{ij}[x_i,x_j]+\sum_{i=1}^n\beta_i x_i\2=\\
		&=\sum_{j=1}^n\alpha_{0j}[x_0,x_j]+\beta_0(x_0\2+[x_0,z]+z\2)+\sum_{1\leq i<j\leq n}\alpha_{ij}[x_i,x_j]+\sum_{i=1}^n\beta_i x_i\2=\\
		&=\sum_{0\leq i<j\leq n}\alpha_{ij}[x_i,x_j]+\sum_{i=0}^n\beta_i x_i\2=r=0,
	\end{align*}
	and the restriction to $\mf m$ of the natural projection $\pi:\mf g\times \mf k/\mf k\2\to \mf g$ is the inverse of $\phi.$
\end{proof}

The same proof, with the omission of the $2$-powers of elements, shows that if $\mf g$ is the restrictification of an ordinary Bloch-Kato Lie algebra, then $\mf g\times \mf k$ is Bloch-Kato (cf. Example \ref{ex:nonBKds}). Moreover, for these Lie algebras there is no distinction between the weak and strong Bloch-Kato properties. 

\begin{cor}
	Let $A$ be a commutative, graded algebra over $k.$ Then the symmetric tensor product $A\wedge k[T]$ is universally Koszul.	\end{cor}

We finally provide another example of a Bloch-Kato restricted Lie algebra that is not the restrictification of an ordinary Lie algebra. 
\begin{exam}\label{ex:x2+y2}
	The Lie algebra $\mf s=\pres{x,y}{x\2+y\2}$ is quadratic, and, being a $1$-relator, its cohomological dimension is $2,$ proving that it is Koszul. Moreover, all the proper subalgebras are free abelian, hence Koszul, which implies that $\mf s$ is Bloch-Kato. 
	
	Equivalently, notice that its quadratic dual admits a presentation as a connected algebra \[A=\pres{\xi,\eta}{\xi\eta,\eta\xi,\xi^2+\eta^2}_{\text{alg}},\] and hence it is universally Koszul as the annihilator of $a\xi+b\eta,$ $(a,b)\in k^2\setminus\{(0,0)\},$ is generated by the degree-$1$ element $b\xi+a\eta$ (see \cite{MPPT}). 
\end{exam}

\subsection{HNN-extensions}
Following Lichtman and Shirvani \cite{hnnLS}, we consider HNN-extensions of graded restricted Lie algebras in characteristic $2.$ 

Let $\mf g$ be a graded restricted Lie algebra over $k,$ and let $\mf h$ be a restricted homogeneous subalgebra. If $n$ is a positive integer, a degree-$n$ derivation of $\mf h$ into $\mf g$ is a $k$-linear map $\phi:\mf h\to \mf g$ such that, for every homogeneous element $x\in \mf h_i,$ one has 
\begin{enumerate}
	\item $\phi(x)\in \mf g_{i+n}$;
	\item $\phi(x\2)=[x,\phi(x)].$ 
\end{enumerate}
It follows from (\ref{brackets2}) of \S \ref{sec:restr} that the usual Leibniz product formula holds true, i.e., for $x,y\in \mf h,$ \[\phi([x,y])=[\phi(x),y]+[x,\phi(y)].\]

The aim of HNN-extensions is that of embedding $\mf g$ into a larger Lie algebra where the derivation $\phi$ extends to an inner derivation.

\begin{defin}
	Let $n$ be a positive integer, and let $\phi:\mf h\to \mf g$ be a degree-$n$ derivation of a homogeneous subalgebra $\mf h$ of a graded restricted Lie algebra $\mf g.$ A graded restricted Lie algebra $\mf e$ together with a homogeneous restricted homomorphism $\iota:\mf g\to \mf e,$ and an element $t\in \mf e$ such that $\ad(t)\vert_{\mf h}=\iota\phi$ is an HNN-extension of $\phi$ if it satisfies the following universal property: For all restricted Lie algebras $\mf m,$ all homogeneous restricted homomorphisms $\jmath:\mf g\to \mf m,$ and all $s\in \mf m$ such that $\ad(s)\vert_{\mf h}=\jmath\phi,$ there exists a unique homogeneous restricted homomorphism $\psi:\mf e\to \mf m$ such that $\psi\iota=\jmath.$
\end{defin}
The HNN-extension exists and is unique up to (unique) isomorphism; we denote it by $\hnn_\phi(\mf g,t).$
An explicit presentation for $\mf e$ is given in terms of a presentation $\mf g=\pres{V}{R}$ as follows
\[\mf e=\pres{V,t}{R, [t,x]+\phi(x):\ x\in \mf h}.\]
It is clear that the additional generator $t,$ which is called the stable letter of the extension, is a homogeneous element of $\mf e$ of the same degree as the derivation $\phi.$
Moreover, \cite{hnnLS} proves that $\mf g$ naturally embeds into the HNN-extension $\mf e.$ 

The proof of Proposition 2.2 by Kochloukova and Mart\'inez P\'erez \cite{cmp} adapts verbatim to the graded, restricted context. 
\begin{thm}[\cite{cmp}]\label{thm:hnn}
	Let $\mf e=\hnn_\phi(\mf g,t)$ be the HNN-extension of a graded restricted Lie $k$-algebra $\mf g$ with respect to a degree-$n$ derivation $\mf h\to \mf g.$ Then, there is a short exact sequence of left $u(\mf e)$-modules
	\[0\to u(\mf e)\otimes _{u(\mf h)}k\overset{\alpha}{\to}u(\mf e)\otimes _{u(\mf g)}k\overset{\beta}{\to}k\to 0 \]
	where, for $x\in u(\mf e),$ $\alpha(x\otimes 1)=xt\otimes 1,$ $\beta(x\otimes 1)=\varepsilon(x),$ and $\varepsilon:u(\mf e)\to k$ is the augmentation map. The maps $\alpha$ and $\beta$ are homogeneous maps of degree $n$ and $0,$ respectively. In particular, by the bigraded Eckmann-Shapiro lemma \cite[Thm. 1.4]{sb}, if $M$ is a graded restricted $\mf e$-module, there is a long exact sequence \[\dots \to H^{ij}(\mf e,M)\to H^{ij}(\mf g,M)\to H^{i,j-n}(\mf h,M)\to H^{i+1,j}(\mf e,M)\to \dots \]
\end{thm}

\begin{rem}
	In the theory of quadratic (restricted) Lie algebras of characteristic $\neq 2,$ \cite{sb_kosz} proves that if $\mf g$ is a quadratic Lie algebra, and $\mf m$ is a maximal proper standard subalgebra of $\mf g,$ with $t\in \mf g_1\setminus \mf m,$ then $\mf g=\hnn_{\phi}(\mf m,t),$ where $\phi:=\ad(x)\vert_{\mf h}:\mf h\to \mf m,$ and $\mf h=\gen{x\in \mf m_1:\ [t,x]\in \mf m}.$ Using this general decomposition, an easy inductive procedure shows that weakly Bloch-Kato Lie algebras of characteristic $\neq 2$ are Bloch-Kato. The bijection established by the restrictification functor also shows that restricted weakly Bloch-Kato Lie algebras of characteristic $\neq 2$ are Bloch-Kato.
	
	Unfortunately, in characteristic two, the above-mentioned decomposition is not always available; for instance, the quadratic Lie algebra $\pres{x,y}{(x+y)\2},$ as well as the Lie algebra $\mf s$ of Example \ref{ex:x2+y2}, does not split as an HNN-extension over $\gen{x}.$
\end{rem}

When $\mf h=\mf g,$ namely when $\phi$ is a derivation of $\mf g,$ then $\mf g$ is clearly a restricted ideal of the HNN-extension $\mf e=\hnn_\phi(\mf g,t),$ and there is a split extension \[0\to \mf g\to \mf e\to \mf k\to 0,\] that is $\mf e=\mf g\rtimes_\phi \mf k$ is a semidirect product.
\subsection{A generalized Witt formula}\label{sec:witt}
Let $\mu:\mathbb N\to \mathbb N$ be the M\"obius function, and consider its mod-$2$ version (see Petrogradsky \cite{restrwitt}) \begin{equation}
	\mu_2(n)=\begin{cases}
		\mu(n),& \text{if }n\text{ is odd,}\\
		2^{s-1}\mu(m),&\text{if }n=2^s m,\text{ and } m\text{ is odd}.
	\end{cases}
\end{equation}
It satisfies \begin{equation}\label{eq:moebius}
	\sum_{d\vert m}(-1)^{d+1}\mu_2(m/d)=\delta_{1,m}.\end{equation}

The \textbf{generalized necklace polynomial} mod-$2$ of degree $m\geq 1$ is defined similarly to the classical one, where the M\"obius function is replaced by $\mu_2,$ namely,
\[M_{2,m}(t)=\frac{1}{m}\sum_{j\vert m}\mu_2(m/j)t^j\in \Z[\![t]\!].\]

Now, recall that, given a graded, connected associative algebra $A$ over a field $k$ such that $\ext^\bullet_A(k,k)$ is finite-dimensional, one defines the \textbf{eigenvalues} of $A$ as the complex numbers $\lambda_i$ such that \[\chi_A(t):=\sum_{i,j\geq 0}\dim\left(\ext_A^{i,j}(k,k)\right)t^j=\prod_{i=1}^n(1+\lambda_i t).\] The left-hand side is the Hilbert series of the $\ext$-algebra of $A$ with respect to the \textit{internal degree}. If $A$ is the restricted envelope of a graded restricted Lie algebra $\mf g,$ then the numbers $\lambda_i$ are called the eigenvalues of $\mf g.$
\begin{thm}
	Let $\mf g$ be a graded, restricted Lie algebra with finite-dimensional cohomology and eigenvalues $\lambda_1,\dots,\lambda_n.$ Then, for all $m\geq 1,$ \[\dim\mf g_m=\sum_{i=1}^nM_{2,m}(\lambda_i).\]
\end{thm}
\begin{proof}
	We adapt the proof of Weigel \cite{weig} to the restricted world. 
	
	Let $A=u(\mf g).$ By the Poincaré-Birkhoff-Witt for $p$-restricted Lie algebras, the Hilbert series of $A$ is \[h_A(t):=\sum_{i\geq 0}\dim (A_i )t^i=\prod_{m\geq 1}\left(\frac{1-t^{mp}}{1-t^m}\right)^{\ell_m}\] where $\ell_m=\dim\mf g_m.$ By Fr\"oberg's identity for internal degrees (see \cite[Prop. 2.2]{weig}), one has $\chi_A(-t)h_A(t)=1$ in $\Z[\![t]\!],$ and for $p=2$ we deduce that \[\prod_{m\geq 1}(1+t^m)^{\ell_m}=\prod_{i=1}^n(1-\lambda_i t)^{-1}\quad \text{in }\Z[\![t]\!].\]
	
	By taking $-\log$ from both sides, 
	\[-\sum_{m\geq 1}\ell_m \log(1+t^m)=\sum_{i=1}^n\log(1-\lambda_it)\]
	and then, using the definition of the logarithm of a power series 
	$\log(1-y)=\sum_{j\geq 1}y^j/j$ for $y\in t\mathbb C[\![t]\!],$ 
	\[\sum_{m\geq 1}\sum_{j\geq 1}\frac{\ell_m}{j}(-1)^{j+1}t^{mj}=\sum_{i=1}^n\sum_{m\geq 1}\frac{(\lambda_i t)^m}{m}.\]
	
	By rewriting the left-hand side as \[\sum_{m\geq 1}\sum_{j\vert m }\ell_{m/j}\frac{(-1)^{j+1}}{j}\ t^m,\] we compare the coefficients of $t^m$ and get, for all $m\geq 1,$ \begin{equation}\label{eq:formula}
		\sum_{j\vert m}\frac{\ell_{m/j}}{m/j}(-1)^{j+1}=\sum_{i=1}^{n}{\lambda_i}^m.
	\end{equation}
	
	The result thus follows from Equation \ref{eq:moebius} and the generalized M\"obius inversion formula. Indeed, if $\alpha(m)=(-1)^{m+1},$ and $\beta(m)=\ell_m/m,$ then the left-hand side of (\ref{eq:formula}) is the Dirichlet convolution $(\alpha\star \beta )(m).$ By (\ref{eq:moebius}), $\mu_2$ is the Dirichlet inverse of $\alpha,$ i.e., $\mu_2\star \alpha\star \beta=\beta,$ and hence \[\ell_m=\frac{1}{m}\sum_{i=1}^n\sum_{d\vert m}\mu_2(m/d)\lambda_i^d=\sum_{i=1}^k M_{2,m}(\lambda_i).\qedhere\]
\end{proof}

For the standard free restricted Lie algebra $\mf f$ on $n$ generators, we recover the restricted Witt's formula \cite[Cor. 2.1]{restrwitt}: $\mf f$ has a single eigenvalue, $\lambda_1=n,$ whence \[\dim(\mf f_m)=\frac{1}{m}\sum_{d\vert m}\mu_2(m/d)n^m.\]

\section{Graphs}
In this section we present the two notions of graphs we need to define T-RAAG and ERA Lie algebras. 
\subsection{Mixed graphs}\label{sec:mixed}
A \textbf{mixed graph} $\Gamma=(V,E,D,o,t)$ consists of a finite non-empty set $V$ of vertices, a set $E$ of $2$-element subsets of $V,$ a subset $D\subset E,$ and two maps $o,t:D\to V,$ respectively called the origin map and the terminus map, such that $o(e)\neq t(e)$ for all $e\in D.$ 
If $e=\{v,w\}\in E,$ then we denote $e$ by $\e {vw}$ if $e\notin D,$ and by $\vec{vw}$ if $e\in D$ and $v=o(e),$ $w=t(e).$ The elements of $D$ are the \textbf{directed edges} of $\Gamma.$ 

In the geometric realization of the graph $\Gamma$ we will denote the edges in $E$ in the following way, depending on their nature:
\[\overline{vw}\in E\setminus D:\quad\begin{tikzpicture}
	\node[fill=black!100,circle,scale=0.3,draw, "$v$"] at (0,0) {};
	\node[fill=black!100,circle,scale=0.3,draw, "$w$"] at (1,0) {};
	
	\draw[thick] (0,0) -- (1,0);
\end{tikzpicture}
\hspace{2cm}\vec{vw}\in D:\quad 
\begin{tikzpicture}
	\node[fill=black!100,circle,scale=0.3,draw, "$v$"] at (0,0) {};
	\node[fill=black!100,circle,scale=0.3,draw, "$w$"] at (1,0) {};
	
	\draw[thick,mid arrow] (0,0) -- (1,0);
\end{tikzpicture}.\]
The pair $(V,E)$ is the underlying simple graph of $\Gamma,$ and is denoted by $\bar \Gamma.$ By abuse of language, we call $\Gamma$ a simple graph whenever $D=\emptyset,$ and we identify $\Gamma$ and $\bar\Gamma.$

A vertex $v\in V$ of $\Gamma$ is said to be \textbf{central} if it is adjacent to all the other vertices, i.e., $\{v,w\}\in E,$ for every $w\in V\setminus\{v\}.$

The mixed graph $\Gamma$ is said to be \textbf{special} if for all $e\in D$ and all $e'\in E,$ if $t(e)\in e',$ then $e'\in D$ with $t(e')=t(e).$ 
Consequently, the set of vertices of $\Gamma$ can be partitioned into the two subsets of positive vertices $V_+$ and of negative vertices $V_-,$ where all the termini of edges in $D$ lie in $V_-.$ If the special graph $\Gamma$ has no isolated vertices, then there is a unique choice for this partition. 

Suppose that $\Gamma$ is a special mixed graph with a given partition $V=V_+\cup V_-$; the assignment $\theta(v)=1$ if $v\in V_-$ and $\theta(v)=0$ if $v\in V_+$ is called a \textbf{signature} of $\Gamma.$ We define the \textbf{cone} $\nabla_\theta(\Gamma)=(V',E',D',o',t')$ as the mixed graph containing $\Gamma$ with \[V'=V\cup \{z\},\quad E'=E\cup\set{\{v,z\}}{v\in V},\] and \[D'=D\cup\set{\{z,v\}}{v\in V, \theta(v)=1},\quad (o',t')\vert_{D'\setminus D}:\{z,v\}\mapsto(z,v).\]
In particular, the tip $z$ of the cone is a central vertex.

\begin{defin}[see \cite{sb_cq}]
	A mixed graph $\Gamma$ is a \textbf{mixed Droms graph} if $\Gamma$ is a special mixed graph without an induced subgraph of the form \[\Lambda_s=	
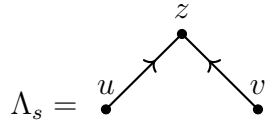
\begin{tikzpicture}
		\node[fill=black!100,circle,scale=0.3,draw, "$u$"] at (0,0) {};
		\node[fill=black!100,circle,scale=0.3,draw, "$z$"] at (1,1) {};
		\node[fill=black!100,circle,scale=0.3,draw, "$v$"] at (2,0) {};
		\draw[thick,mid arrow] (0,0) -- (1,1);
		\draw[thick,mid arrow] (2,0) -- (1,1);
	\end{tikzpicture} \] and its underlying simple graph $(V,E)$ does not contain any induced subgraph that is either a square graph $C_4$ or a path on four vertices $P_4.$ 
\end{defin}
Equivalently, by \cite[Prop. 2.14]{sb_qw}, $\Gamma$ is a mixed Droms graph if and only if it belongs to the smallest class of special mixed graphs containing the singleton mixed graph with any signature, and that is closed under the coning construction and the disjoint union. By a simple Droms graph, we will mean a mixed Droms graph with no directed edges.

Any mixed graph $\Lambda=(V_1,E_1,D_1,o\vert_{D_1},t\vert_{D_1})$ -- where $V_1\subseteq V,$ $E_1=E\cap \mathcal P(V_1),$ and $D_1=D\cap E_1$ -- is called an induced subgraph of the mixed graph $\Gamma=(V,E,D,o,t),$ and we also say that $\Lambda$ is spanned by $V_1.$ If $n\geq 0$ is any integer, an $n$-clique of $\Gamma$ is a set of $n$ distinct pairwise adjacent vertices. A clique of a special mixed graph has at most a single negative vertex, i.e., the subgraph it spans (also called a clique of $\Gamma$) is either a simplicial graph, or has $n-1$ directed edges with common terminus. If $\bar\Gamma$ is a simple graph, then we denote by $C_{\bar\Gamma}(t)$ the \textbf{clique polynomial} of $\bar\Gamma,$ i.e., 
\[C_{\bar\Gamma}(t)=\sum_{n\geq 0}c_nt^n,\] where $c_n$ is the number of $n$-cliques of $\bar\Gamma.$

\subsection{Labelled graphs}

If $\Gamma=(V,E)$ is a simple graph (or equivalently, a mixed graph with $D=\emptyset$), a function $\theta:V\to \F_2=\{0,1\}$ is called a \textbf{$2$-labelling} of $\Gamma.$
In the simplicial realization of $\Gamma$ vertices labelled with the unit $1\in \F_2$ are denoted by a filled circle $\bullet,$ while the others by an empty circle $\circ.$ 

\begin{defin}\label{def:labelDroms}
	A graph $\Gamma$ with a $2$-labelling $\theta:V\to \F_2$ is a \textbf{Droms labelled graph} if it is a simple Droms graph, and it does not contain any of the following labelled graphs as induced subgraphs: 
	\[\begin{tikzpicture}
		\node[fill=black!100,circle,scale=0.3,draw] at (0,0) {};
		\node[circle,scale=0.3,draw] at (1,0) {};
		\node[fill=black!100,circle,scale=0.3,draw] at (2,0) {};
		\draw[thick] (0,0) -- (0.945,0) ;
		\draw[thick] (1.055,0) -- (2,0);
	\end{tikzpicture},\qquad \begin{tikzpicture}
		\node[fill=black!100,circle,scale=0.3,draw] at (0,0) {};
		\node[circle,scale=0.3,draw] at (1,0) {};
		\node[circle,scale=0.3,draw] at (2,0) {};
		\draw[thick] (0,0) -- (0.945,0);
		\draw[thick] (1.055,0) -- (1.945,0);
	\end{tikzpicture}.
	\]
\end{defin}

\begin{lem}\label{lem:labelDroms}
	Let $\Gamma$ be a Droms labelled graph, and let $\Lambda$ be a connected induced subgraph. Then, either $\theta\vert_{V(\Lambda)}\equiv 0,$ or there exists a central vertex $v$ of $\Lambda$ with $\theta(v)=1.$
\end{lem}
\begin{proof}
	Since $\Lambda$ is a connected Droms graph, it contains a central vertex $v.$ 
	Suppose that all the central vertices of $\Lambda$ have value $0.$ If $u$ is not a central vertex, then either $\theta(u)=0,$ or there exists a vertex $w\neq v,u$ such that the induced subgraph spanned by $\{v,u,w\}$ is one of the path graphs of Definition \ref{def:labelDroms}, which is impossible since $\Gamma$ is a Droms labelled graph.
\end{proof}

In particular, the class of Droms labelled graphs is the smallest class of $2$-labelled graphs containing arbitrarily $2$-labelled discrete vertices, and that is closed under taking \begin{enumerate}
	\item disjoint unions, and
	\item cones, with the exception that the tip can only be labelled by $0$ if the basis of the cone has no vertex labelled by $1.$
\end{enumerate}

\section{T-RAAG Lie algebras}
We first consider restricted Lie algebras associated to mixed graphs.
\begin{defin}
	Let $\Gamma=(V,E,D,o,t)$ be a mixed graph. 
	We define the T-RAAG Lie $k$-algebra associated to $\Gamma$ as the restricted Lie algebra over the field $k$		
	\[\mf a(\Gamma)=\pres{V}{\begin{matrix}
			[v,w]:& \e{vw}\in E\setminus D\\
			[v,w]+v\2:&\vec{vw}\in D
	\end{matrix}}.\]
	
\end{defin}
Observe that, due to the identity (\ref{brackets2}), the T-RAAG Lie algebra $\mf a(\Gamma)$ admits a presentation that does not explicitly involve the Lie brackets:
\begin{equation}\label{eq:2pres}
	\pres{V}{\begin{matrix}
			(v+w)\2+v\2+w\2:& \e{vw}\in E\\ 
			(v+w)\2+w\2:& \vec{vw}\in D
\end{matrix}}.\end{equation}

We identify the vertices of $\Gamma$ with the corresponding elements of $\mf a(\Gamma),$ which are henceforth called the \textbf{canonical generators} of the T-RAAG Lie algebra.

\begin{rem}
	When $D$ is empty, i.e., when the graph is simple, we also call the restricted Lie algebra $\mf a(\Gamma)$ the \textbf{RAAG Lie algebra} associated to $\Gamma.$ The ordinary subalgebra of $\mf a(\Gamma)$ generated by its elements of degree $1$ is called the ordinary RAAG Lie algebra, and is denoted by $\li_\Gamma.$ The restrictification of $\li_\Gamma$ is precisely $\mf a(\Gamma)$; in particular, the universal enveloping algebra of $\li_\Gamma$ is isomorphic to the restricted enveloping algebra $u(\mf a(\Gamma))$ in a natural way.
	
	If $D=\emptyset,$ one can similarly define a $p$-restricted Lie algebra $\mf a(\Gamma,p)$ over a field $k$ of characteristic $p>0.$ For $k=\F_p,$ it is the associated graded restricted Lie algebra of the pro-$p$ completion $\hat A$ of the right-angled Artin group $A(\Gamma)$ with respect to the Jennings-Zassenhaus filtration. Moreover, the restricted $p$-enveloping algebra of $\mf a(\Gamma,p)$ is the associated graded algebra of the $\omega$-adic filtration of the complete group algebra $\F_p[\![\hat A]\!]=\invlim_{ U\lhd_o \hat A}\hat A/U,$ having $\omega$ as its augmentation ideal. Nevertheless, the $p$-analogue of the presentation (\ref{eq:2pres}) does not produce $\mf a(\Gamma,p)$ if the graph contains some edges. 
\end{rem}

\begin{question}
	Consider a group $G$ and the augmentation ideal $\omega$ of the group ring $kG.$ We denote by $\gr kG$ the associated graded algebra of the $\omega$-adic filtration of $kG.$ 
	Moreover, if $\gr ^\gamma G$ is the Lie ring associated to the lower central series of $G,$ put $\gr G=\gr ^\gamma G\otimes k.$ 
	
	Let $\Gamma$ be a mixed graph, and let $A(\Gamma)$ be its associated twisted right-angled Artin group.
	
	(Q1) Is $u(\mf a(\Gamma))$ isomorphic to $\gr kA(\Gamma)$?
	
	(Q2) Is $\gr A(\Gamma)$ isomorphic to $\mf a(\Gamma)$?
	
	We believe that this question admits a positive answer, and we will address it in a forthcoming work. 
	
\end{question}

If $\Gamma$ is a special mixed graph with a partition $V=V_+\cup V_-,$ the signature $\theta:V\to k$ extends to a restricted Lie homomorphism $\theta:\mf a(\Gamma)\to k,$ where $k$ is seen as an elementary abelian, standard, restricted Lie algebra, i.e., $k\simeq \mf k/\mf k\2.$ 
For all edges $e=\{v,w\}\in E,$ one may rewrite the associated relation of $\mf a(\Gamma)$ as $[v,w]=\theta(v)w\2+\theta(w)v\2,$ which provides an equivalent and compact presentation for the T-RAAG of a special mixed graph $\Gamma$:
\[\mf a(\Gamma)=\pres{V}{[v,w]+\theta(v)w\2+\theta(w)v\2:\ \{v,w\}\in E}.\]

\begin{prop}\label{prop:ind}
	If $\Lambda$ is an induced subgraph of a special mixed graph $\Gamma,$ then $\mf a(\Lambda)$ naturally embeds into $\mf a(\Gamma).$ 
\end{prop}
\begin{proof}
	First of all notice that the natural map $\mf a(\Lambda)\to \mf a(\Gamma)$ induced by the inclusion $\Lambda\to \Gamma$ is a restricted homomorphism. 
	
	We argue by induction on the number of vertices of $\Gamma=(V,E,D,o,t).$ It is enough to suppose that it is connected and that $\Lambda\neq \Gamma.$ The case when $\Gamma$ consists of a single vertex is obvious, so suppose that $\Gamma$ has more than one vertex. If $v$ is a vertex in $\Gamma$ but not in $\Lambda,$ let ${\Gamma_v}$ be the induced subgraph of $\Gamma$ with vertex set $V\setminus\{v\}.$ By induction, $\mf a(\Lambda)$ naturally embeds into $\mf a({\Gamma_v}).$ 
	
	If $v\in V_+,$ the map $\mf a(\Gamma)\to \mf a({\Gamma_v})$ defined on generators by $v\mapsto 0$ and $x\in V({\Gamma_v}) \mapsto x$ is a restricted Lie map, and is a left inverse to the natural map $\mf a({\Gamma_v})\to\mf a(\Gamma)$ constructed above, which is thus injective.
	
	On the other hand, if $v\in V_-,$ and $W$ is the set of vertices that are adjacent to $v,$ one can split $\mf a(\Gamma)$ into the HNN-extension of $\mf a({\Gamma_v})$ with respect to the derivation $\phi:\mf b\to \mf a({\Gamma_v})$ defined by sending each canonical generator $w$ to $w\2,$ where $\mf b=\mf a(\Pi)$ is the subalgebra of $\mf a({\Gamma_v})$ generated by $W,$ and $\Pi$ is the subgraph induced by $W.$ It follows that $\mf a({\Gamma_v})$ embeds into $\mf a(\Gamma)$ by \cite{hnnLS}.
\end{proof}

\begin{cor}\label{cor:TRAAGkosz}
	If $\Gamma$ is a special mixed graph, then $\mf a(\Gamma)$ is a Koszul restricted Lie algebra. 
\end{cor}
\begin{proof}
	We argue by induction on the number $n$ of vertices of $\Gamma.$ If $n=1,$ then  $u(\mf a(\Gamma))=k[T]$ is clearly a Koszul algebra, and hence so is $\mf a(\Gamma)$ by definition. 
	Suppose now that $n>1.$ If $V_-=\emptyset,$ then $\mf a(\Gamma)$ is Koszul by Fr\"oberg \cite{frob} (see also \cite{sb_kosz} for a less combinatorial proof).
	On the other hand, if $v\in V_-\neq \emptyset,$ then one recovers the same HNN-decomposition of $\mf a(\Gamma)$ as in Proposition \ref{prop:ind}, in terms of subalgebras $\mf a(\Gamma_v)$ and $\mf b$ which are Koszul by induction. By Theorem \ref{thm:hnn} we conclude (see \cite[Lem. 2.2]{sb_kosz}).
\end{proof}

\begin{cor}\label{cor:TRAAGcoho}
	If $\Gamma$ is a special mixed graph, then the cohomology ring of the T-RAAG Lie algebra of $\mf a(\Gamma)$ is the commutative connected $k$-algebra presented by 
	\[
	H^\bullet(\mf a(\Gamma),k)=\pres{\xi_i}{\begin{matrix}
			\xi_{i}\xi_{j}:&\{v_i,v_j\}\notin E\\
			\xi_{p}^2+\xi_p\xi_q:& (v_p,v_q)=(o(e),t(e)),\ e\in D
	\end{matrix}}_\text{alg},
	\]
	where $\xi_i$ denotes the dual basis element of the canonical generator $v_i\in V.$
	
	In particular, the cohomological dimension 
	$\cd\mf a(\Gamma)$ equals the clique number of $\Gamma,$ that is the maximal integer $n$ such that $\Gamma$ contains an $n$-clique.
\end{cor}
\begin{proof}
	Since $\mf a=\mf a(\Gamma)$ is a Koszul Lie algebra, its cohomology ring is isomorphic to the quotient of the tensor algebra over $\mf a_1^\ast=\hom_k(\mf a_1,k)$ by the ideal generated by the elements $\alpha\in (\mf a_1\otimes \mf a_1)^\ast$ such that $\alpha(v\otimes w-w\otimes v+\theta(v)w\otimes w+\theta(w)v\otimes v) =0$ for all $\{v,w\}\in E,$ where we have identified $\mf a_1^\ast\otimes \mf a_1^\ast$ with $(\mf a_1\otimes \mf a_1)^\ast.$ Let $(\xi_i)$ be the dual basis of $\mf a_1^\ast$ relative to $(v_i),$ i.e., $\xi_i(v_j)=\delta_{ij}.$
	The second assertion follows from the fact that $\xi_{i}\xi_{j},$ vanishes in $H^\bullet(\mf a)$ when $\{v_i,v_j\}\notin E,$ $i\neq j,$ and $\xi^3=\xi_i^2\xi_j=\xi_i\xi_j^2=0$ when $(v_i,v_j)\in (o,t)(D),$ and $\xi\in H^1(\mf a).$
\end{proof}

By Example \ref{ex:1dim}(2), the T-RAAG associated to a special mixed graph is torsion-free for it has finite cohomological dimension. On the other hand, if $\Gamma$ is not a special graph, then $\mf a(\Gamma)$ might have torsion.

\begin{exam}
	Consider a non-special complete graph $\Delta$ with an oriented cycle, e.g., 
    \begin{center}
        $V=\{v_1,\dots,v_n\},$ $E=\set{\{v_i,v_j\}}{i\neq j,\ v_i,v_j\in V},$  $D=\set{\{v_i,v_{i+1}\}}{i=1,\dots,n},$
    \end{center} and $(o,t)\{v_i,v_{i+1}\}=(v_i,v_{i+1})$ (indices mod $n$).
	In the associated T-RAAG $\mf a(\Delta),$ one has \[\left(\sum_{i=1}^n v_i\right)\2=\sum_{i=1}^nv_i\2+\sum_{i< j}[v_i,v_j]=2\sum_{i=1}^nv_i\2=0,\]
	that is, the sum of all canonical generators is a torsion element.
\end{exam}

The existence of torsion elements depends on the ground field. 
\begin{exam}
	Consider the restricted Lie algebra $\mf g=\pres{x,y}{[x,y]+x\2, [x,y]+y\2}.$ By abusing the notation of T-RAAG Lie algebras, $\mf g$ can be seen as the Lie algebra associated to the following directed graph 
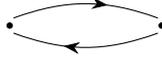
\[\begin{tikzpicture}[
		>={Stealth[length=6pt]},   
		node/.style={circle,draw,minimum size=8mm,inner sep=1pt}
		]
		\draw[fill=black] (0,0) circle (1pt);
		\draw[fill=black] (2,0) circle (1pt);
		
		\draw[mid arrow] (0.05,0.1) to[bend left=20]  (1.95,0.1);
		\draw[mid arrow] (1.95,-0.1) to[bend left=20]  (0.05,-0.1);
	\end{tikzpicture}.\]
	
	
	Since $\mf g_2=kx\2,$ if $\alpha,\beta\in k,$ then $(\alpha x+\beta y)\2=0$ implies $\alpha^2+\beta^2+\alpha\beta=0,$ which has a non-trivial solution iff $k$ contains $\F_4.$  Notice that for such $k,$ if $\alpha\in k$ satisfies $\alpha^2+\alpha+1=0,$ then $\mf g=\pres{x',y'}{{x'}\2,{y'}\2}$ is a free product of elementary abelian Lie $k$-algebras, where $x'=\alpha x+y,$ and $y'=x+\alpha y.$ 
\end{exam}

Unlike the simplicial case, if $\Gamma$ is a mixed graph and $\Lambda$ is an induced subgraph, then $\mf a(\Lambda)$ might not be a retract of $\mf a(\Gamma)$ in a natural way.
\begin{exam}
	Consider the single-edge graph $\Gamma=$
	\begin{tikzpicture}
		\node[fill=black!100,circle,scale=0.3,draw, "$u$"] at (0,0) {};
		\node[fill=black!100,circle,scale=0.3,draw, "$v$"] at (1,0) {};
		\draw[thick,mid arrow] (0,0) -- (1,0);
	\end{tikzpicture}. 
	
	The quotient of the T-RAAG Lie algebra $\mf a(\Gamma)=\pres{u,v}{[u,v]+u\2}$ by the restricted ideal generated by $v$ is elementary abelian, and hence not a T-RAAG.
\end{exam}
Nevertheless, we can avoid the situation of the example if we only consider subgraphs containing all the negative vertices of $\Gamma.$ 
\begin{lem}\label{lem:twistedRetract}
	Let $\Gamma$ be a special mixed graph. Let $X$ be a set of vertices of $\Gamma$ not containing the terminus of any edge, i.e., $X\cap t(D)=\emptyset.$ Then, $\mf a(\Gamma)/(X)$ is the T-RAAG Lie algebra associated to the induced graph of $\Gamma$ spanned by $V\setminus X.$
\end{lem}
\begin{proof}
	Let $\Lambda$ be the graph spanned by $V\setminus X.$ 
	We prove that the natural map  $\iota:\mf a(\Lambda)\to \mf a(\Gamma)$ admits a retraction, namely, a left inverse. Let $\sigma(v)=0$ for $v\in X,$ and $\sigma(v)=v$ for any vertex $v$ of $\Lambda.$ A computation of $\sigma$ on the relations shows that $\sigma$ extends to a well-defined restricted homomorphism, which is thus a retraction of $\iota.$
\end{proof}
The same proof shows that $X$ can also contain a negative vertex $v,$ provided that it also contains the whole \textit{star} of $v,$ namely, the set of its adjacent vertices. 

In the unidirected case, the graph is determined by its associated RAAG Lie algebra, namely, if $\mf a(\Gamma)\simeq \mf a(\Gamma'),$ then the graphs $\Gamma$ and $\Gamma'$ are isomorphic (see Kim and Roush \cite{kimroush}, and Droms \cite{droms_iso} for the group theoretic case). As for twisted right-angled Artin groups (cf. \cite{sb_droms}), this is no longer true in the Lie algebra case.
\begin{exam}\label{ex:nonrigid}
	Let \[\vec \Sigma_2=\begin{tikzpicture}
		\node[fill=black!100,circle,scale=0.3,draw,"$x$"] at (0,-1) {};
		\node[fill=black!100,circle,scale=0.3,draw,"$y_1$"] at (-1,0) {};
		\node[fill=black!100,circle,scale=0.3,draw,"$y_2$"] at (1,0) {};
		\draw[thick,mid arrow] (0,-1) -- (1,0);
		\draw[thick,mid arrow] (0,-1) -- (-1,0);
	\end{tikzpicture}\]
	and consider the associated T-RAAG Lie algebra over $k.$
	
	Let $y'=y_1+y_2.$ One has $[y',x]=[y_1,x]+[y_2,x]=2x\2=0.$ It follows that the assignment $\phi(x)=x',$ $\phi(y_1)=y_1',$ $\phi(y_2)=y_1'+y'$ defines an isomorphism between $\mf a(\vec{\Sigma}_2)$ and the restricted Lie algebra \[\pres{x',y'_1,y'}{[x',y'],[x',y_1']={x'}\2},\]
	which is the T-RAAG $\mf a(\Gamma_2),$ where  \[\Gamma_2=\begin{tikzpicture}
		\node[fill=black!100,circle,scale=0.3,draw,"$x$"] at (0,-1) {};
		\node[fill=black!100,circle,scale=0.3,draw,"$y_1'$"] at (-1,0) {};
		\node[fill=black!100,circle,scale=0.3,draw,"$y'$"] at (1,0) {};
		\draw[thick] (0,-1) -- (1,0);
		\draw[thick,mid arrow] (0,-1) -- (-1,0);
	\end{tikzpicture}.\]
	In particular, we see that the mixed graph is not determined by the isomorphism class of the associated T-RAAG. 
	
	This procedure clearly applies to any \textit{directed star graph} 
	$\vec \Sigma_{n}$ with vertex set  $V=\{0,1,\dots,n\},$ and edge sets $E=D=\set{\{0,i\}}{i=1,\dots,n},$ with $o(e)=0$ for all $e\in D.$ The graph $\Gamma_n=(V,E,\{0,1\},o,t)$ has the same simplicial edge-set as $\vec \Sigma_{n},$ and a single directed edge, say $D'=\{\{0,1\}\}.$
\end{exam}

\subsection{Hilbert series of T-RAAGs}
Let $\Gamma=(V,E,D,o,t)$ be a special mixed graph. Denote by $\Gamma^\circ$ the simple graph with vertex set $V$ and edge-set $E'=E\setminus D.$ We compute the Hilbert series of $\mf a(\Gamma)$ from that of the RAAG Lie algebra $\mf a(\Gamma^\circ).$

In fact, one clearly has a surjective restricted homomorphism $\pi:\mf a(\Gamma^\circ)\to \mf a(\Gamma),$ which is an isomorphism in degree $1.$ Since these T-RAAG Lie algebras are Koszul, the inflation $\pi^\ast:H^\bullet(\mf a(\Gamma))\to H^\bullet(\mf a(\Gamma^\circ))$ is surjective, and we need to compute its kernel.

For this, fix an ordering on $V$ and let $\mf X$ be the set of cliques of $\Gamma$ properly containing the terminus of some edges in $D,$ i.e., \[\mf X=\set{\mathcal O=\{v_1,\dots,v_n\}\subseteq V}{\mathcal O\cap V_-\neq \emptyset,\ \abs{\mathcal O}=n\geq 2}.\]
Let $\mathcal O\in \mf X.$ As $\Gamma$ is a special graph, $\abs{\mathcal O\cap V_-}=1,$ and we let $v_{\mathcal O}$ be the minimal vertex of $\mathcal O\cap V_+$ (with respect to the chosen ordering). 

By Corollary \ref{cor:TRAAGcoho} the ideal $\ker \pi^\ast$ is generated in $H^\bullet(\mf a(\Gamma))$ by the elements $(v^\ast)^2,$ where $v\in o(D)$ and $v^\ast$ denotes the Kronecker dual of the canonical basis element $v,$ i.e., $v^\ast(w)=1$ if $w=v,$ and $v^\ast(w)=0$ if $w\in V\setminus\{v\}.$ 
Moreover, \[\ker\pi^\ast=\sum_{\mathcal O\in \mf X}k\cdot\left[ v_{\mathcal O}^\ast\prod _{v\in \mathcal O\cap V_+}v^\ast\right].\]

Indeed, if $v_1,\dots,v_m$ are pairwise-adjacent vertices, then there are only two possible cases: 

(1) If $v_1$ is not the origin of any edge $e\in D,$ then $(v_1^\ast)^2$ already vanishes in $H^\bullet(\mf a(\Gamma)).$ 

(2) If $v_1$ is the origin of an edge $e\in D,$ with $v=t(e),$ then $(v_1^\ast)^2v_2^\ast\dots v_m^\ast=v_1^\ast v^\ast v_2^\ast\dots v_m^\ast$ does not vanish in $H^\bullet(\mf a(\Gamma))$ only if each $v_i$ is adjacent to $v$; if this is the case, then $(v_1^\ast)^2v_2^\ast\dots v_m^\ast=v_1^\ast (v_2^\ast)^2\dots v_m^\ast=\dots=v_1^\ast v_2^\ast\dots (v_m^\ast)^2,$ and $\{v,v_1,\dots,v_m\}\in \mf X$ is a clique containing a negative element.

On the other hand, if $v_1,\dots,v_m$ are not all pairwise adjacent, then $(v_1^\ast)^2v_2^\ast\dots v_m^\ast=0.$

\begin{exam}
	(1) Consider the single-edge $\Gamma=$
	\begin{tikzpicture}
		\node[fill=black!100,circle,scale=0.3,draw, "$u$"] at (0,0) {};
		\node[fill=black!100,circle,scale=0.3,draw, "$v$"] at (1,0) {};
		\draw[thick,mid arrow] (0,0) -- (1,0);
	\end{tikzpicture}; one has $\Gamma^\circ=$ 	\begin{tikzpicture}
		\node[fill=black!100,circle,scale=0.3,draw, "$u$"] at (0,0) {};
		\node[fill=black!100,circle,scale=0.3,draw, "$v$"] at (1,0) {};
	\end{tikzpicture},
	and thus \[H^\bullet(\mf a(\Gamma))=k[u^\ast,v^\ast]/((u^\ast)^2+u^\ast v^\ast,(v^\ast)^2),\]and \[H^\bullet(\mf a(\Gamma^\circ))=k[u^\ast,v^\ast]/((u^\ast)^2,u^\ast v^\ast,(v^\ast)^2),\]
	which gives $\ker\pi^\ast=k\cdot (u^\ast)^2.$

	(2) Consider the graphs
	
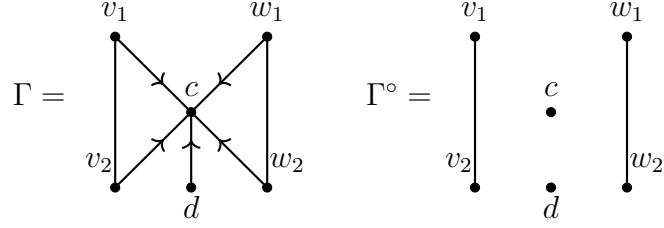
\begin{figure}[H]
		\centering
		\begin{tikzpicture}
			
			\node["$\Gamma=$"] at (-2,-0.2) {};
			\node[fill=black!100,circle,scale=0.3,draw, "$v_1$"] (a) at (-1,1) {};
			\node[fill=black!100,circle,scale=0.3,draw, "$v_2\ \ \ \ $"] (b) at (-1,-1) {};
			\node[fill=black!100,circle,scale=0.3,draw, "$c$"] (c) at (0,0) {};
			\node[fill=black!100,circle,scale=0.3,draw, "$w_1$"] (d) at (1,1) {};
			\node[fill=black!100,circle,scale=0.3,draw, "$\ \ \ \ \ w_2$"] (e) at (1,-1) {};
			\node[fill=black!100,circle,scale=0.3,draw] (f) at (0,-1) {};
			
			\node["$d$"] at (0,-1.7) {};
			
			\draw[thick] (a) -- (b);
			\draw[thick,mid arrow] (a) -- (c);
			\draw[thick,mid arrow] (b) -- (c);
			\draw[thick,mid arrow] (e) -- (c);
			\draw[thick,mid arrow] (d) -- (c);
			\draw[thick] (d) -- (e);
			\draw[thick,mid arrow] (f) -- (c);
		\end{tikzpicture}
		\quad and \quad
		\begin{tikzpicture}
			
			\node["$\Gamma^\circ=$"] at (-2,-0.2) {};
			\node[fill=black!100,circle,scale=0.3,draw, "$v_1$"] (a) at (-1,1) {};
			\node[fill=black!100,circle,scale=0.3,draw, "$v_2\ \ \ \ $"] (b) at (-1,-1) {};
			\node[fill=black!100,circle,scale=0.3,draw, "$c$"] (c) at (0,0) {};
			\node[fill=black!100,circle,scale=0.3,draw, "$w_1$"] (d) at (1,1) {};
			\node[fill=black!100,circle,scale=0.3,draw, "$\ \ \ \ \ w_2$"] (e) at (1,-1) {};
			\node[fill=black!100,circle,scale=0.3,draw] (f) at (0,-1) {};
			
			\node["$d$"] at (0,-1.7) {};
			
			\draw[thick] (a) -- (b);
			\draw[thick] (d) -- (e);
		\end{tikzpicture}.
	\end{figure}
	
	The cohomology rings of $\mf a(\Gamma)$ and of $\mf a(\Gamma^\circ)$ are quotients of the polynomial algebra \[k[v_i^\ast,w_i^\ast,c^\ast,d^\ast: i=1,2]\] by the ideals generated respectively by $\Omega_\Gamma$ and $\Omega_{\Gamma^\circ},$ where \begin{gather*}
		\Omega_\Gamma=\spn_k\{(c^\ast)^2,c^\ast v_i^\ast+(v_i^\ast)^2,c^\ast w_i^\ast+(w_i^\ast)^2,c^\ast d^\ast+(d^\ast)^2,d^\ast v_i^\ast,d^\ast w_i^\ast,v_i^\ast w_j^\ast:\ i,j=1,2\},\\
		\Omega_{\Gamma^\circ}=\spn_k\{(c^\ast)^2,(v_i^\ast)^2,(w_i^\ast)^2,(d^\ast)^2,d^\ast v_i^\ast,d^\ast w_i^\ast,v_i^\ast w_j^\ast, c^\ast v_i^\ast,c^\ast w_i^\ast,c^\ast d^\ast:\ i,j=1,2\}.
	\end{gather*}
	
	The homogeneous components of the kernel $I=\ker\pi^\ast$ are 
	\begin{align*}
		I_j=\begin{cases}
			0,& j<2\text{, or }j>3\\
			k\cdot(d^\ast)^2+\sum_{i=1}^2k\cdot (v_i^\ast)^2+k\cdot (w_i^\ast)^2, &j=2\\
			k\cdot (v_1^\ast)^2v_2^\ast+k\cdot (w_1^\ast)^2w_2^\ast, &j=3
		\end{cases}.
	\end{align*}
	
	Indeed, $(v_1^\ast)^2v_2=(v_2^\ast)^2,$ and $(v_1^\ast)^3=(v_1^\ast)^2 c^\ast=v_1^\ast (c^\ast)^2=0,$ etc.		
\end{exam}

\begin{prop}
	Let $\Gamma$ be a special mixed graph. For $i\geq 2,$ let $n_i$ be the number of $i$-cliques containing a negative vertex. Then, the Poincaré series of $\mf a(\Gamma),$ i.e., the Hilbert series of $H^\bullet(\mf a(\Gamma)),$ is \[\sum_{i\geq 0}\dim H^i(\mf a(\Gamma))t^i=\operatorname{C_{\Gamma^\circ}}(z)+\sum_{i\geq 2}n_iz^i\]
	where $C_\Delta(t)$ denotes the clique polynomial of the simplicial graph $\Delta.$
\end{prop}
\begin{proof}
	Consider the natural surjection $\pi^\ast:H^\bullet(\mf a(\Gamma))\to H^\bullet(\mf a(\Gamma^\circ)).$ Recall that the restricted cohomology of $\mf a(\Gamma^\circ)$ is isomorphic to the Chevalley-Eilenberg cohomology of the ordinary Lie subalgebra $\li_{\Gamma^\circ}$ of $\mf a(\Gamma^\circ)$ generated by its elements of degree $1$ (cf. the restrictification functor in \cite{sb_kosz}).  It follows that \[\sum_{i\geq 0}\dim H^i(\mf a(\Gamma))t^i=\sum_{i\geq 0}\left(\dim H^i(\mf a(\Gamma^\circ))+\dim(\ker\pi^\ast)_i\right)z^i.\]
	From the computations performed above, we get $\dim (\ker\pi^\ast)_i=n_i,$ which gives the desired formula, as the Poincaré polynomial of an ordinary RAAG Lie algebra is the clique polynomial of the defining graph. 
\end{proof}
The next result thus follows at once. 
\begin{thm}
	Let $\Gamma$ be a special mixed graph, and let $\mf a(\Gamma)$ be the associated T-RAAG Lie algebra. Then, the Poincaré series of $\mf a(\Gamma)$ is the clique polynomial of the underlying simple graph $\bar\Gamma.$
\end{thm}

Since the restricted envelope $u(\mf a(\Gamma))$ is Koszul, it follows from Fr\"oberg's formula that the Hilbert series of $u(\mf a(\Gamma))$ is

\begin{equation}\label{eq:hilbu}
	\sum_{i\geq 0}\dim (u(\mf a(\Gamma))_i) t^i=C_{\bar\Gamma}(-t)^{-1}.
\end{equation}

By the main theorem of Petrogradsky \cite{restrwitt}, the Hilbert series of the T-RAAG $\mf a(\Gamma)$ is 
\[\sum_{i\geq 1}\dim\mf (a(\Gamma)_i)t^i=\mathcal{L}_2(\phi):=\sum_{n\geq 1}\frac{\mu_2(n)}{n}\ln \phi(t^n)\]
where $\phi$ is the formal power series of the Hilbert series (\ref{eq:hilbu}) of $u(\mf a(\Gamma)),$ and $\mu_2$ is the mod-$2$ M\"obius function described in \S \ref{sec:witt}. 

\section{Subalgebras of T-RAAGs}
In this section, we provide a complete characterization of Bloch-Kato T-RAAGs in terms of the structure of the defining graph. 

To begin with, we exclude non-special graphs from the discussion.

\begin{prop}\label{prop:nonspecial}
	Let $\Gamma$ be a non-special mixed graph. Then, $\mf a(\Gamma)$ contains a non\--quad\-rat\-ic, standard subalgebra. 
\end{prop}
\begin{proof}
	Since $\Gamma$ is not special, it contains an induced subgraph $\Lambda$ on three vertices $u,v,w$ such that $\vec{uv}\in D$ and either $\vec{vw}\in D$ or $\overline{vw}\in E\setminus D.$ We put $\theta=1$ in the first case, and $\theta=0$ otherwise, so that $[v,w]=\theta v\2.$ 
	Consider the subalgebra $\mf b$ of $\mf a(\Gamma)$ generated by the vertices of $\Lambda.$ We can assume that $\mf b$ is quadratic, and hence isomorphic to $\mf a(\Lambda).$ 
	
	Let $\mf m$ be the subalgebra of $\mf a(\Lambda)$ generated by $u+w$ and $v.$ Since $[u+w,v]=u\2+\theta v\2,$ the quadratic cover $\q\mf m$ is a free restricted Lie algebra, and hence it is a proper extension of $\mf m$ as $[u+w,v\2]=0.$
\end{proof}

\begin{exam}\label{ex:nonBKgraphs}
	Let $\Gamma=(V,E,D,o,t)$ be a special mixed graph, and let $\mf a(\Gamma)$ be the associated T-RAAG over $k.$
	
	\begin{enumerate}
		\item\label{ex:LambdaS} Let $\Gamma=\Lambda_s,$ i.e., $V=\{1,2,3\}$ and $E=D=\{\vec{12},\vec{32}\}.$ Denote the canonical generator of $\mf a(\Lambda_s)$ corresponding to $i\in V$ by $x_i.$ 
		
				%
		
		Consider the subalgebra $\mc b$ generated by $x_1+x_3$ and $x_2.$ Then, \[[x_2,x_1+x_3]=x_1\2+x_3\2\qquad\mbox{and}\qquad(x_1+x_3)\2=x_1\2+x_3\2+[x_1,x_3],\] proving that $\q\mf b$ is free. Since for $i\in \{1,3\}$ \[[x_2\2,x_i]=[x_2,[x_2,x_i]]=[x_2,x_i\2]=[[x_2,x_i],x_i]=[x_i\2,x_i]=0,\] it follows that $[x_2\2,x_1+x_3]=0,$ and hence $\mf b$ is not quadratic. Notice that, for RAAG Lie algebras associated to simple graphs, the existence of non-quadratic standard subalgebras generated by $2$ elements is prevented by \cite{sb_kosz}.
		\item 	Let $\bar\Gamma$ be the square graph, i.e. $V=\{1,2,3,4\}$ and $E=\{\{1,2\},\{2,3\},\{3,4\},\{1,4\}\}$.
		\begin{enumerate}
			\item\label{ex:square} If $D=\emptyset,$ then the subalgebra $\mf b$ of $\mf a(\Gamma)$ generated by $x_1+x_2,x_3,x_4$ is not quadratic. Indeed, the only degree-$2$ relation of $\mf b$ is $[x_3,x_4]=0,$ and one has $[[x_1+x_2,x_3],x_4]=0.$
			\item Suppose that $\vec{12}\in D.$ Since the graph $\Gamma$ is special, also $\vec{32}\in D.$ Then the subgraph $\Lambda$ on $\{1,2,3\}$ is the same mixed graph of (\ref{ex:LambdaS}), and $\mf a(\Lambda_s)$ is a standard subalgebra of $\mf a(\Gamma).$
		\end{enumerate}
		\item Suppose that $\bar\Gamma=P_4$ is the path of length $3,$ i.e., $V=\{1,2,3,4\}$ and  $E=\{\{2,3\},\{3,4\},\{4,1\}\}.$ 
		\begin{enumerate}
			\item If $D=\emptyset,$ then, the same computations as in (\ref{ex:square}) prove that the subalgebra $\mf b$ of $\mf a(\Gamma)$ generated by $x_1+x_2,x_3,x_4$ is not quadratic. 
			\item If $D$ is nonempty, then either $\Gamma$ contains an induced $\Lambda_s,$ or $V_-\subseteq\{1,2\}.$ In the latter case, let $\theta:V\to k$ be the associated signature. 
			One has \begin{align*}
				[[x_1+x_2,x_3],x_4]&=[[x_1,x_3]+\theta(x_2)x_3\2,x_4]=\\
				&=[[x_1,x_4],x_3]+[x_1,[x_3,x_4]]+\theta(x_2)[x_3\2,x_4]=\\
				&=\theta(x_1)[x_4\2,x_3]+0+0=0.
			\end{align*}

			On the other hand, the only degree-$2$ relation involving the elements  $x_1+x_2,x_3,x_4$ are $[x_1,x_2]=x_i\2,$ for $\theta(x_i)=0$ ($i=1,2$), and  $[x_3,x_4]=0,$ proving that the standard subalgebra generated by those elements is not quadratic.
		\end{enumerate}
	\end{enumerate}
\end{exam}
It follows that if $\Gamma$ is not a Droms mixed graph, then $\mf a(\Gamma)$ contains a non-quadratic standard subalgebra.

\begin{lem}\label{lem:simpleBK}
	If $\Gamma$ is a simple graph, i.e., $D=\emptyset,$ then all standard restricted subalgebras of $\mf a(\Gamma)$ are RAAG Lie algebras if and only if $\Gamma$ is a Droms graph. 
\end{lem}
\begin{proof}
	Let $\li_\Gamma$ be the ordinary RAAG Lie algebra over $k.$ The restrictification of $\li_\Gamma$ is $\mf a(\Gamma),$ which is Bloch-Kato precisely when so is $\li_\Gamma.$ Hence, the result follows from \cite[Ex.~4.4]{sb}.
\end{proof}

\subsection{The prime field \texorpdfstring{$\F_2$}{\F_2}}


For $k=\F_2$ the prime field of characteristic $2,$ the situation is identical to that of the group theoretic one given in \cite{sb_droms}.

\begin{prop}
	If $\Gamma$ is a mixed graph and $\mf a(\Gamma)$ is the associated T-RAAG restricted Lie algebra over the prime field $\mathbb F_2,$ then the following are equivalent: 
	\begin{enumerate}
		\item $\Gamma$ is a mixed Droms graph.
		\item All standard restricted subalgebras of $\mf a(\Gamma)$ are T-RAAGs.
		\item $\mf a(\Gamma)$ is weakly Bloch-Kato, i.e., all standard restricted subalgebras of $\mf a(\Gamma)$ are quad- ratic restricted Lie algebras.
		\item $\mf a(\Gamma)$ is Bloch-Kato, i.e., all standard restricted subalgebras of $\mf a(\Gamma)$ are Koszul restricted Lie algebras.
		\item The cohomology ring of $\mf a(\Gamma)$ is universally Koszul.
	\end{enumerate}
\end{prop}
\begin{proof}
	The equivalence (4)$\iff$(5) is clear, and obviously (4) implies (3). Example \ref{ex:nonBKgraphs}, together with Proposition \ref{prop:nonspecial}, shows that (3) implies (1). Moreover, by Corollary \ref{cor:TRAAGkosz}, (2) implies (4). It remains to prove that (1) implies (2).
	The proof is based on that of the group theoretic analogue \cite{sb_droms}. Suppose that $\Gamma$ is a Droms mixed graph. We argue by induction on the number of its vertices $\abs V.$ 
	
	The case $\abs{V}=1$ is trivial, so suppose that $\abs V>1.$ If $\Gamma$ is disconnected, say  $\Gamma=\Gamma_1\sqcup \Gamma_2,$ then $\mf a(\Gamma)$ is the free product $\mf a(\Gamma)=\mf a(\Gamma_1)\amalg\mf a(\Gamma_2)$ of restricted Lie algebras. In particular, the cohomology ring of $\mf a(\Gamma)$ is the direct sum of those of the $\mf a(\Gamma_i)$'s, and hence, by induction, it is universally Koszul, proving that $\mf a(\Gamma)$ is Bloch-Kato. From the Kurosh subalgebra theorem \cite{sb}, $(3)$ follows. 
	
	Now suppose that $\Gamma$ is connected, and hence a cone $\nabla_{\theta'}(\Gamma'),$ where $\theta'$ is the signature of $\Gamma'$ and $v$ is the tip of the cone, i.e., $V=V(\Gamma')\cup \{v\}.$ The unique signature $\theta$ for $V$ is the extension of $\theta'$ with $\theta(v)=0.$
	Hence, $\mf a(\Gamma)$ is the semidirect product $\gen v\rtimes \mf a(\Gamma').$ By induction, we suppose that all the standard subalgebras $\mf c$ of $\mf a(\Gamma')$ are T-RAAG restricted Lie algebras with signature $\theta'\vert_{\mf c}.$ Let $\pi:\mf a(\Gamma)\to \mf a(\Gamma')$ be the natural projection.
	If $\mf b$ is a standard subalgebra of $\mf a(\Gamma),$ then the image $\pi(\mf b)$ in $\mf a(\Gamma')$ is a standard subalgebra of $\mf a(\Gamma'),$ and hence a T-RAAG Lie algebra by induction, i.e., there exists a mixed graph $\Lambda=(V(\Lambda),E(\Lambda),D(\Lambda),o_\Lambda,t_\Lambda)$ such that $\pi(\mf b)\simeq \mf a(\Lambda),$ with signature $\theta'\vert_{\pi(\mf b)}.$ Put  $V(\Lambda)=\{v_1,\dots,v_r\},$ and let $x_i\in \mf b$ such that $\pi(x_i)=v_i.$ 
	For all $i,$ since $\pi\vert_{\mf a(\Gamma')}$ is the identity map, there exists a scalar $\alpha_i\in \F_2$ such that $v_i+x_i=\alpha_i v.$ 
	We need to show that the assignment $\sigma(v_i)=v_i+\alpha_iv$ extends to a restricted Lie map $\pi(\mf b)\to \mf b$ giving a split exact sequence 
	\begin{equation}\label{eq:ses}
		\begin{tikzcd}
			0 \arrow{r} & \mf b \cap \gen v \arrow{r} & \mf b \arrow{r}{\pi} & \pi(\mf b) \arrow{r} \arrow[bend left=33, dashed]{l}{\sigma} & 0 .
		\end{tikzcd}
	\end{equation}
	
	Let $v_i$ and $v_j$ be two different vertices of $\Lambda.$ There are two cases to consider. 
	
	(i) If $\e{v_iv_j}\in E(\Lambda),$ then $[v_i,v_j]=0$ and $\theta'(v_i)=\theta'(v_j)=0.$ Hence, $[v_i,v]=[v_j,v]=0,$ proving that $[\sigma(v_i),\sigma(v_j)]=0.$ 
	
	(ii) If $\vec{v_iv_j}\in D(\Lambda),$ then $\theta'(v_i)=0$ and $\theta'(v_j)=1.$ 
	One has \[[v_i+\alpha_iv,v_j+\alpha_j v]=v_i\2+\alpha_i v\2.\]
	Since $\alpha_i\in \F_2,$ one has $\alpha_i^2=\alpha_i,$ and hence $[\sigma(v_i),\sigma(v_j)]=\sigma(v_i)\2.$
	
	Therefore, $\sigma:\pi(\mf b)\to \mf b$ is a homomorphism, and $\pi\sigma$ is the identity on $\pi(\mf b),$ proving that the exact sequence (\ref{eq:ses}) splits. We deduce that $\mf b=(\mf b\cap \gen v)\rtimes \pi(\mf b)$ is a T-RAAG Lie algebra with canonical generators $x_i$ ($i=1,\dots,r$), and possibly $x\in \mf b_1\cap \gen v.$ 
	Finally, the map $\tilde\theta(x_i)=\theta'(v_i),$ $\tilde\theta(x)=0,$ for $x\in \mf b_1\cap \gen v,$ is a signature for the defining graph of $\mf b$ and is the restriction of $\theta.$ Since $\mf b$ is a T-RAAG Lie algebra, we deduce that $\mf a(\Gamma)$ is Bloch-Kato, and hence its cohomology ring is universally Koszul.
\end{proof}



	
	\subsection{Non-prime field}
	
	In the following example, we show an unexpected behaviour of T-RAAGs over larger fields: if $k\neq \F_2,$ then the T-RAAG Lie algebra of a particular Droms graph with $3$ vertices is not Bloch-Kato, and its cohomology is not universally Koszul. This is essentially due to the trivial fact that the map $\mf a\to\mf a$ of a torsion-free, abelian, restricted Lie algebra $\mf a$ over a field $k$ sending each element to its $2$-power is not linear, unless $k=\F_2.$ Indeed, if $\alpha\in k\setminus \F_2,$ and $x\in\mf a$ is non-zero, then $(\alpha x)\2=\alpha^2 x\2\neq \alpha x\2.$
	\begin{exam}\label{ex:F4}Let $k$ be a field of characteristic two with at least $4$ elements, and let $\alpha\in k$ such that $\alpha^2\neq 0.$
		Consider the special graph $e$ consisting of two adjacent vertices $v_1,v_2,$ where $V_+=\{v_1\},$ and $\Gamma=\nabla_\theta(e)$; then the associated T-RAAG restricted Lie algebra over $k$ is \[\mf a=\pres{v,v_1,v_2}{[v_1,v],[v_2,v_1]+v_1\2,[v_2,v]+v\2}.\]
		
		Let $\mf b$ be its standard subalgebra generated by $v_1+\alpha v$ and $v_2.$ 
		Because,   $[v_2,v_1+\alpha v]=v_1\2+\alpha v\2$ and $(v_1+\alpha v)\2=v_1\2+\alpha^2v\2$ are linearly independent, one has $\mf b_2=k\cdot v\2\oplus k\cdot v_1\2\oplus k\cdot v_2\2=\mf a_2.$ Therefore, $\mf b_2$ has the same dimension as the degree-$2$ component of the free restricted Lie algebra on two generators, and we deduce that $\q\mf b$ is free. Nevertheless, the equality $[v_2,(v_1+\alpha v)\2]=0$ implies that $\mf b$ cannot be free, and $\q\mf b\neq \mf b.$ Notice that, since $(v_1+\alpha v)\2=v_1\2+\alpha ^2v\2,$ and $[v_1,+\alpha v,v_2]=v_1\2+\alpha v\2,$ one has $\mf b\cap \gen v=\gen{v\2},$ i.e., it is zero in degree $1$ and contains $v\2.$
		
		Since $\mf a$ is Koszul, its cohomology is the associative algebra with presentation \[H^\bullet(\mf a,k)=\pres{a,b,c}{ab+ba,ac+ca,bc+cb,c^2,ac+a^2,bc+b^2}_\text{alg},\]
		where $(a,b,c)$ is the dual basis to the basis $(v,v_1,v_2).$ Although we already know that $H^\bullet(\mf a,k)$ is not universally Koszul (since $\mf a$ is not Bloch-Kato), we prove this explicitly. 
		Consider the ideal $I$ generated by the element $x=\alpha a+b,$ i.e., $I$ is the ideal generated by the elements $f\in H^1(\mf a,k)=\mf a_1^\ast$ such that $f\vert_{\mf b_1}\equiv0.$  
		
		\noindent The elements $xa,xb,xc$ are linearly independent, whence $I$ has no relation of degree $2$ as an $H^\bullet(\mf a,k)$-module. Nevertheless, $x\cdot(\alpha(1+\alpha)^{-1}ab+b^2)=0,$ and hence $\tor_{1,3}^A(I,k)\neq 0,$ which means that $I$ is not a quadratic $A$-module (of degree $1$).
	\end{exam}

	\begin{lem}\label{lem:single-}
		Let $\Gamma'$ be a simple Droms graph, and let $Y$ be a mixed graph consisting of a single vertex $y,$ which is a negative vertex not lying in $\Gamma'.$ Consider the special cone $\Gamma=\nabla_\theta(\Gamma'\sqcup Y)$ on the mixed graph $\Gamma'\sqcup Y,$ where $\theta:V(\Gamma'\sqcup Y)\to k$ is the signature with $\theta(y)=1$ and $\theta\vert_{V(\Gamma')}\equiv 0.$ Then, the T-RAAG Lie algebra $\mf a(\Gamma)$ over $k$ is Bloch-Kato. 
	\end{lem}
	\begin{proof}
		By Lemma \ref{lem:simpleBK}, all the standard subalgebras of $\mf a(\Gamma_v)$ are T-RAAGs, where $v$ is the tip of the cone $\Gamma=\nabla(\Gamma_v),$ and $\Gamma_v=\Gamma'\sqcup Y.$
		
		Let $\mf m$ be a standard subalgebra of $\mf a(\Gamma).$ If $\mf m$ is contained in $\mf a(\Gamma_v),$ then $\mf m$ is a T-RAAG Lie algebra. So, assume that $v$ appears with a non-trivial coefficient in the linear expression of some element of degree $1$ of $\mf m.$ By Lemma \ref{lem:twistedRetract}, $\mf a(\Gamma_v)$ is the quotient $\mf a(\Gamma)/(v)$; let $\pi$ be the corresponding projection.
		
		We can choose the generators $v+u_0,u_1,\dots,u_n$ of $\mf m$ so that the $u_i$'s are elements of $\mf a(\Gamma_v).$ Consider the Lie subalgebra $\mf h=\gen{u_0,\dots,u_n}$ of the restricted RAAG Lie algebra $\mf a(\Gamma_v),$ i.e., $\mf h=\pi(\mf m).$ Clearly, if $u_0$ lies in the linear span of the $u_i$'s for $i\geq 1,$ then $\mf m\simeq \mf h\times \mf k.$ We can thus focus on the contrary. 
		
		Since Droms theorem holds true for restricted RAAG Lie algebras over arbitrary fields (Lemma \ref{lem:simpleBK}), $\mf h$ is a restricted RAAG Lie algebra. Moreover, up to a Gauss elimination procedure, we can assume that $y$ only appears in one of the $u_i$'s, say $u_{i_0}.$ In particular, by Kurosh theorem \cite{sb}, $\mf h$ is the free product of the free restricted Lie algebra generated by $u_{i_0}$ and its subalgebra $\gen{u_0,\dots,\hat{u}_{i_0},\dots,u_n},$ where $\hat{u}_{i_0}$ means that the generator $u_{i_0}$ is omitted.
		
		
		Consider the map $f:\mf h\to \mf m$ given by $f(u_0)=v+u_{0}$ and $f(u_i)=u_i$ for $i>0.$ We need to check well-definedness. After that, notice that $f$ is the inverse of $\pi\vert_{\mf m},$ i.e., $\mf h\simeq \mf m.$ 
		
		If $i_0=0,$ then $\mf h=\gen{u_0}\amalg \gen{u_1,\dots,u_n},$ and $f$ is well-defined by the universal property of the free product. 
		
		If $i_0>0,$ then $f$ is well-defined as $[v,u_i]=0$ for all $i\neq i_0,$ and no relation involves $u_{i_0}.$ 
	\end{proof}
	
	\begin{cor}
		Let $k\neq \F_2$ be a field of characteristic two, and let $\Gamma$ be a mixed graph. Then the T-RAAG Lie algebra $\mf a(\Gamma)$ over $k$ is Bloch-Kato if and only if $\Gamma$ is a Droms mixed graph, and all the directed edges in any connected component of $\Gamma$ have a common origin.
	\end{cor}
	\begin{proof}
		First, suppose that $\mf a(\Gamma)$ is Bloch-Kato. Then $\Gamma$ is a Droms mixed graph by Proposition \ref{prop:nonspecial} and Example \ref{ex:nonBKgraphs}.
		
		Let us assume, without loss of generality, that $\Gamma$ is connected and has two different directed edges $e_1,e_2\in D.$ If, by contradiction $o(e_1)\neq o(e_2),$ then there are two cases to consider. 
		
		(1) If $t(e_1)=t(e_2),$ then either $\Lambda_s$ or the graph of Example \ref{ex:F4} is induced in $\Gamma,$ proving that $\mf a(\Gamma)$ is not Bloch-Kato.
		
		(2) If $t(e_1)\neq t(e_2),$ then $e_1\cap e_2=\emptyset.$ Moreover, there exists a positive vertex $v$ that is adjacent to all the other vertices of $\Gamma.$ We may assume that $v\notin e_2,$ so that the induced subgraph spanned by $v$ and $e_2$ is the graph described in Example \ref{ex:F4}, contradicting the assumption on $\mf a(\Gamma).$ 
		
		Now assume that $\Gamma$ is a Droms mixed graph, and all directed edges in any connected component have a common origin. We argue by induction as usual. If $\Gamma$ is disconnected, by induction, $\mf a(\Gamma)$ is a free product of Bloch-Kato Lie algebras, and hence it is Bloch-Kato. 
		
		We can thus suppose that $\Gamma$ is connected, and hence there is a positive vertex $v$ that is central in $\Gamma.$ Clearly, if $D\neq \emptyset,$ then $v$ is the common origin of all the edges $e\in D.$ Since $\Gamma$ is a special mixed graph, the graph $\Gamma_v$ induced by the vertices $\neq v$ is a disjoint union of a simple Droms graph and a discrete set of negative vertices, i.e., $\Gamma$ is obtained by attaching a star $\vec \Sigma_{n+1}$ at $v$ to the subgraph spanned by all the positive vertices. In view of Example \ref{ex:nonrigid}, we can assume that $\Gamma$ has a single negative vertex, and hence, by induction and Lemma \ref{lem:single-}, we deduce that $\mf a(\Gamma)$ is Bloch-Kato.
	\end{proof}
	
	\begin{exam}
		The graph $\Gamma$ with realization	
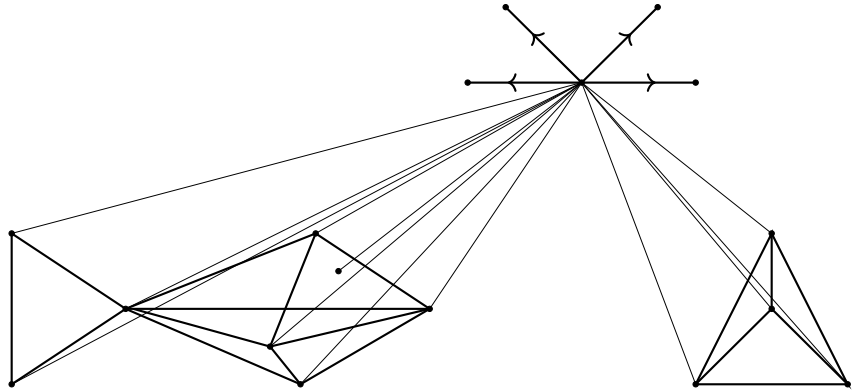
\begin{figure}[H]
			\centering
			\begin{tikzpicture}
				\draw[fill=black] (0,0) circle (1pt);
				\draw[fill=black] (0,2) circle (1pt);
				\draw[fill=black] (1.5,1) circle (1pt);
				\draw[fill=black] (3.4,0.5) circle (1pt);
				\draw[fill=black] (3.8,0) circle (1pt);
				\draw[fill=black] (4.3,1.5) circle (1pt);
				\draw[fill=black] (4,2) circle (1pt);
				\draw[fill=black] (5.5,1) circle (1pt);

				\draw[thick] (0,0) -- (0, 2);
				\draw[thick] (0,0) -- (1.5,1);
				\draw[thick] (1.5,1) -- (3.8,0);
				\draw[thick] (1.5,1) -- (0,2);
				\draw[thick] (1.5,1) -- (4,2);
				\draw[thick] (3.4,0.5) -- (4,2);
				\draw[thick] (5.5,1) -- (4,2);
				\draw[thick] (5.5,1) -- (3.8,0);
				\draw[thick] (5.5,1) -- (1.5,1);
				\draw[thick] (3.4,0.5) -- (1.6,1);
				\draw[thick] (3.4,0.5) -- (5.5,1);
				\draw[thick] (3.4,0.5) -- (3.8,0);
				
				
				
				\draw[fill=black] (7.5,4) circle (1pt);
				\draw[fill=black] (9,4) circle (1pt);
				\draw[fill=black] (8.5,5) circle (1pt);
				\draw[fill=black] (6,4) circle (1pt);
				\draw[fill=black] (6.5,5) circle (1pt);
				
				\draw[thick, mid arrow] (7.5,4) -- (9,4);
				\draw[thick, mid arrow] (7.5,4) -- (8.5,5);
				\draw[thick, mid arrow] (7.5,4) -- (6,4);
				\draw[thick, mid arrow] (7.5,4) -- (6.5,5);
				
				\draw[fill=black] (10,1) circle (1pt);
				\draw[fill=black] (10,2) circle (1pt);
				\draw[fill=black] (9,0) circle (1pt);
				\draw[fill=black] (11,0) circle (1pt);
				
				\draw[thick] (10,1)--(10,2)--(9,0)--(11,0);
				\draw[thick] (9,0)--(10,1)--(11,0)--(10,2);

				\draw[line width=0.2pt] (7.5,4) -- (0,0);
				\draw[line width=0.2pt] (7.5,4) -- (0,2);
				\draw[line width=0.2pt] (7.5,4) -- (1.5,1);
				\draw[line width=0.2pt] (7.5,4) -- (3.4,0.5);
				\draw[line width=0.2pt] (7.5,4) -- (3.8,0);
				\draw[line width=0.2pt] (7.5,4) -- (4.3,1.5);
				\draw[line width=0.2pt] (7.5,4) -- (4,2);
				\draw[line width=0.2pt] (7.5,4) -- (5.5,1);
				\draw[line width=0.2pt] (7.5,4) -- (9,4);
				\draw[line width=0.2pt] (7.5,4) -- (8.5,5);
				\draw[line width=0.2pt] (7.5,4) -- (6,4);
				\draw[line width=0.2pt] (7.5,4) -- (6.5,5);
				\draw[line width=0.2pt] (7.5,4) -- (10,1);
				\draw[line width=0.2pt] (7.5,4) -- (10,2);
				\draw[line width=0.2pt] (7.5,4) -- (9,0);
				\draw[line width=0.2pt] (7.5,4) -- (11,0);
				
			\end{tikzpicture}
		\end{figure}
		\noindent is obtained by taking the cone of the disjoint union of two simple Droms graphs (the fish-shaped graph $\Phi,$ and the complete graph $K_4$ on $4$ vertices), and then attaching a directed star $\vec \Sigma_4$ to the tip of the cone, i.e., $\Gamma=\nabla(\Phi\sqcup K_4\sqcup 4N),$ where $4N$ is the discrete mixed graph consisting of four negative vertices. The associated T-RAAG Lie algebra over a field of characteristic two is Bloch-Kato, and all of its standard subalgebras are T-RAAGs.
	\end{exam}

	\section{RACG Lie algebras}
	For a simple graph $\Gamma=(V,E)$ we define the \textbf{right-angled Coxeter} \textbf{Lie algebra}, or RACG Lie algebra:
	\[\mf c(\Gamma)=\pres{V}{u\2,\ [v,w]:\ u\in V,\ \{v,w\}\in E}.\]
	
	It can be seen as the quotient of the T-RAAG Lie algebra on the mixed graph  $\Gamma=(V,E,D,o,t)$ with $D=\emptyset,$ with respect to the restricted ideal generated by $v\2,$ for all canonical generators $v.$ The ordinary Lie subalgebra of $\mf c(\Gamma)$ generated by its canonical generators is the ordinary RAAG Lie algebra $\li_\Gamma.$
	
	\begin{exam}
		Let $\Gamma$ be complete. Then, $\mf c(\Gamma)=\pres{v_1,\dots,v_n}{v_i\2,[v_i,v_j]}$ is elementary abelian, and hence $u(\mf c(\Gamma))$ is the exterior algebra on an $n$-dimensional vector space. In particular, $\mf c(\Gamma)$ is Koszul, and its cohomology is the symmetric algebra on the dual vector space $c(\Gamma)_1^\ast.$ Notice that the Chevalley-Eilenberg cohomology of the underlying ordinary Lie algebra of $\mf c(\Gamma)$ is the exterior algebra on $\mf c(\Gamma)^\ast.$
	\end{exam}

	It follows from Lemma \ref{lem:directbk} that if $\mf c(\Gamma)$ is Bloch-Kato, then so is $\mf c(\nabla(\Gamma))\simeq \mf c(\Gamma)\times \mf k/\mf k\2.$ 
	
	For RACGs, the Bloch-Kato property is not equivalent to the fact that all standard subalgebras are RACG. For instance, if $\Gamma$ is discrete with two vertices, then the subalgebra generated by the sum of the canonical generators is not elementary abelian, as it should be if it were a RACG (on a singleton graph). 
	For this we introduce a new class of restricted Lie algebras interpolating between RAAG and RACG Lie algebras.
	
	If $\Gamma$ is a graph, and $\theta:V\to \F_2$ is a vertex labelling of $\Gamma,$ the \textbf{extended right-angled} Lie algebra (or ERA) over $k$ associated to the pair $(\Gamma,\theta)$ is the restricted Lie $k$-algebra with presentation \[\mf e(\Gamma,\theta)=\pres{V}{\theta(v)v\2, [v_1,v_2]:\ v\in V,\ \{v_1,v_2\}\in E}.\]
	
	For instance, if $\theta\equiv 1\in \F_2,$ then $\mf e(\Gamma,1)=\mf c(\Gamma).$ On the contrary, when $\theta\equiv 0,$ then $\mf e(\Gamma,0)$ is the (restricted) RAAG Lie algebra $\mf a(\Gamma).$

	\begin{prop}\label{prop:ERAembed}
		Let $\Gamma=(V,E)$ be a simple graph and $\theta:V\to \F_2$ a $2$-labelling of $\Gamma.$ Let $\mf e(\Gamma)$ be the associated ERA Lie $k$-algebra.
		\begin{enumerate}
			\item If $W\subset V,$ then the subalgebra of $\mf e(\Gamma)$ generated by $W$ is the ERA Lie algebra associated to the induced subgraph spanned by $W,$ with signature $\theta\vert_W.$ 
			\item The ERA Lie algebra $\mf e(\Gamma)$ is Koszul.
		\end{enumerate} 
	\end{prop}
	\begin{proof}
		Let $W$ be a proper subset of $V.$ Clearly, the map induced by $v\in W\mapsto v$ and $v\in V\setminus W\mapsto 0$ is a retraction of that induced by the inclusion $W\hookrightarrow V$; hence (1) follows.
		
		We argue by induction on the number of vertices.
		If $\Gamma$ is complete, then $\mf e(\Gamma)$ is the abelian restricted Lie algebra $k^n\times \mf k^m,$ where $n=\abs{\theta^{-1}(1_{\F_2})}.$ In particular, $H^\bullet(\mf e(\Gamma))$ is the symmetric tensor product of a symmetric algebra and an exterior algebra, which is thus Koszul. Moreover, the subalgebra of $e(\Gamma)$ generated by $W$ is the ERA Lie algebra on the induced subgraph of $\Gamma$ spanned by $W.$ 
		
		So, suppose $\Gamma$ is not complete, and let $v\in V\setminus W.$ Denote by $\Gamma_v$ the induced subgraph of $\Gamma$ spanned by the vertices $\neq v.$ By induction, $\mf e(\Gamma_v)$ is a Koszul restricted Lie algebra.
		
		If $v$ is a central vertex of $\Gamma,$ then $\mf e(\Gamma)\simeq \mf e(\Gamma_v)\times \gen{v},$ proving that $\mf e(\Gamma_v)$ naturally embeds into $\mf e(\Gamma).$ Moreover, $H^\bullet(\mf e(\Gamma))$ is the symmetric tensor product of $H^\bullet(\mf e(\Gamma_v))$ and $H^\bullet(\gen{v})$ --- which are generated in degree $1$ by induction---, implying that $\mf e(\Gamma)$ is Koszul by \cite{pp}. 
		
		If $v$ is not a central vertex, let $\Delta$ be the induced subgraph of $\Gamma$ spanned by the link of $v,$ i.e., $V(\Delta)$ consists of the elements $v'\in V\setminus \{v\}$ such that $\{v,v'\}\in E.$ Let $w$ be a vertex of $\Gamma_v$ that does not lie in $\Delta.$ 
		By induction, $\mf e(\Delta)$ naturally embeds into $\mf e(\Gamma_v).$ Moreover, the map $\phi:\mf e(\Delta)\to \mf e(\Gamma_v)$ defined on the canonical generators by $\phi(u)=[w,u]$ is a derivation of degree $1,$ and the associated HNN-extension is isomorphic with $\mf e(\Gamma).$ This proves both that $\mf e(\Gamma_v)$ naturally embeds into $\mf e(\Gamma),$ and that, by induction and Theorem \ref{thm:hnn}, $\mf e(\Gamma)$ is Koszul.
	\end{proof}
	
	This allows us to compute the cohomology ring of ERA Lie algebras. 
	
	\begin{cor}
		Let $\Gamma=(V,E)$ be a finite simple graph, and let $\theta:V\to \F_2\subseteq k$ be a vertex $2$-labelling of $\Gamma.$
		Then, \[H^\bullet(\mf e(\Gamma,\theta),k)\simeq k[V^\ast]/(v^\ast w^\ast, (u^\ast)^2:\ \{v,w\}\notin E,\ v\neq w,\ u\in V\text{ s.t. } \theta(u)=0).\]
	\end{cor}
	
	In particular, for $\theta\equiv 1,$ we get the cohomology ring of RACG Lie algebras.
	
	\begin{cor}
		Let $\Gamma=(V,E)$ be a simple graph. Then, the cohomology ring of $\mf c(\Gamma)$  is the Stanley-Reisner ring of $\Gamma,$ i.e., \[H^\bullet(\mf c(\Gamma),k)=\mf c(\Gamma)^!\simeq k[\Gamma]:= k[V^\ast]/(v^\ast w^\ast:\ \{v,w\}\notin E, v\neq w).\]
	\end{cor}
	
	Since the RACG Lie algebras (of non-empty graphs) contain non-trivial torsion elements, their cohomological dimension is infinite. 
	
	\begin{exam}
		Consider the discrete graph $\Gamma=\bullet\quad\circ,$ and the associated ERA Lie algebra $\mf g=\pres{x,y}{x\2}\simeq \mf k/\mf k\2\amalg\mf k$ over a field $k$ of characteristic two. Its cohomology is the Koszul dual $\mf g^!=k[\xi,\eta]/(\eta^2,\xi\eta),$ whose Hilbert series is \[h_{\mf g^!}(t)=1+2t+t^2+t^3+\dots,\] which
		can be given in terms of Fibonacci numbers, namely, if $f_i$ denotes the $i$th Fibonacci number, with $f_0=f_1=1,$ $f_{i+2}=f_{i+1}+f_i$ ($i\geq 0$), then \[h_{\mf g^!}(-t)=\left(\sum_{i\geq 0} f_{i+1}t^i\right)^{-1}=\frac{1-t-t^2}{1+t}.\]
		Indeed, one can prove by induction that $-f_{n-1} +\sum_{i=1}^n(-1)^{n-i}f_i=0$ for all $n\geq 2,$ whence the following identity in $\Z[\![t]\!]$ holds: 
		\[(1-2t+t^2-t^3+t^4-\dots)(f_1+f_2t+f_3t^2+\dots)=1.\]
	\end{exam}
	
	\section{Subalgebras of RACG Lie algebras}
	The strategy for proving a Droms-type theorem for ERAs is the same as for T-RAAG Lie algebras.
	\begin{exam}\label{ex:nonBKRACG}We proceed similarly to Example \ref{ex:nonBKgraphs}. 
		\begin{enumerate}
			\item 	Let $\Gamma$ be the square graph, i.e., 
            \begin{center}
                $V=\{1,2,3,4\}$ and $E=\{\{1,2\},\{2,3\},\{3,4\},\{1,4\}\}.$
            \end{center} Then the subalgebra $\mf b$ of $\mf e(\Gamma)$ generated by $x_1+x_2,x_3,x_4$ is not quadratic (here $x_i$ is the canonical generator corresponding to the vertex $i$ of $\Gamma$). Indeed, the only degree-$2$ relations of $\mf b$ are $[x_3,x_4]=\theta(3)x_3\2=\theta(4)x_4\2=0,$ and one has $[[x_1+x_2,x_3],x_4]=0.$

			\item Suppose that $\Gamma=P_4$ is the path of length $3,$ i.e.,
            \begin{center}
                $V=\{1,2,3,4\}$ and  $E=\{\{2,3\},\{3,4\},\{4,1\}\}.$ 
            \end{center}
			As in (1), the subalgebra $\mf b$ of $\mf e(\Gamma)$ generated by $x_1+x_2,x_3,x_4$ is not quadratic.
		\end{enumerate}
		
	\end{exam}
	\begin{exam}\label{ex:nonBKERA}
		Consider the labelled graph 
		\[\Gamma=\begin{tikzpicture}
			\node[fill=black!100,circle,scale=0.3,draw,"$x$"] at (0,0) {};
			\node[circle,scale=0.3,draw,"$y$"] at (1,0) {};
			\node[circle,scale=0.3,draw,"$z$"] at (2,0) {};
			\draw[thick] (0,0) -- (0.945,0);
			\draw[thick] (1.055,0) -- (1.945,0);
		\end{tikzpicture},
		\] where $\theta(x)=1,$ and $\theta(y)=\theta(z)=0$:
		\[\mf e(\Gamma,\theta)=\pres{x,y,z}{x\2,[x,y],[y,z]}.\]
		The subalgebra $\mf h$ of $\mf e(\Gamma,\theta)$ generated by $x+y$ and $z$ is not quadratic. Indeed, its quadratic cover is free as the three elements $[x+y,z]=[x,z],$ $(x+y)\2=y\2,$ and $z\2$ are linearly independent. Since $[(x+y)\2,z]=0,$ we deduce that $\mf h$ is not quadratic. So, we get another example where the direct product of a Bloch-Kato Lie algebra -- the ERA Lie algebra on $x,z$ -- with a free abelian one is not Bloch-Kato itself. 
		In particular, for an ERA Lie algebra, the Bloch-Kato property cannot be checked just by looking at the defining unadorned graph.
		
		The same argument shows that the ERA Lie algebra associated to the graph \[\Gamma=\begin{tikzpicture}
			\node[fill=black!100,circle,scale=0.3,draw,"$x$"] at (0,0) {};
			\node[circle,scale=0.3,draw,"$y$"] at (1,0) {};
			\node[fill=black!100,circle,scale=0.3,draw,"$z$"] at (2,0) {};
			\draw[thick] (0,0) -- (0.945,0);
			\draw[thick] (1.055,0) -- (2,0);
		\end{tikzpicture}\] with $\theta(x)=\theta(z)=1,$ and $\theta(y)=0,$ contains a non-quadratic standard subalgebra. (See Conca \cite[Rem. 3.2]{conca}, where the Koszul dual of this Lie algebra is proven not to be universally Koszul).
	\end{exam}

	\begin{thm}
		Let $\Gamma$ be a graph with a $2$-labelling map $\theta:V\to \F_2.$ Then the following are equivalent: 
		\begin{enumerate}
			\item $\Gamma$ is a Droms $2$-labelled graph.
			\item The graph $\Gamma$ is Droms, and, for all connected induced subgraphs $\Lambda$ of $\Gamma,$ either $\theta\vert_{V(\Lambda)}\equiv 0,$ or there exists a central vertex $v$ of $\Lambda$ with $\theta(v)=1$.
			\item The ERA Lie algebra $\mf e(\Gamma)$ is Bloch-Kato.
			\item The ERA Lie algebra $\mf e(\Gamma)$ is weakly Bloch-Kato.
			\item All the standard subalgebras of the ERA Lie algebra $\mf e(\Gamma)$ are ERA Lie algebras.
		\end{enumerate}
	\end{thm}
	\begin{proof}
		The implication (1)$\implies$(2) is Lemma \ref{lem:labelDroms}.

		Clearly, (5) implies both (3) and (4) by Proposition \ref{prop:ERAembed}(2).
		If $\Gamma$ is not a Droms labelled graph, then $\mf e(\Gamma)$ contains a non-quadratic standard subalgebra by Examples \ref{ex:nonBKRACG}, \ref{ex:nonBKERA}. This proves that any one of (3)-(5) implies (1). 
		
		It remains to show that (2) implies (5). 
		
		We argue by induction on the number of vertices of $\Gamma.$ The base case being trivial, we suppose that $\Gamma$ has at least two vertices.
		
		Let $\Gamma$ be disconnected with one connected component $\Gamma_1,$ and disjoint decomposition $\Gamma=\Gamma_1\sqcup \Gamma_2.$ By induction, both the $\mf e(\Gamma_i)$'s satisfy (2) and (3), and $\mf e(\Gamma)=\mf e(\Gamma_1)\amalg\mf e(\Gamma_2).$ 
		It thus follows from the Kurosh subalgebra theorem \cite{sb} that if $\mf m$ is a standard restricted subalgebra of $\mf e(\Gamma),$ then there exist graphs $\Lambda_1$ and $\Lambda_2,$ maps $\theta_i:V(\Lambda_i)\to \F_2,$ and a free Lie algebra $\mf f,$ such that $\mf m=\mf f\amalg\mf e(\Lambda_1,\theta_1)\amalg\mf e(\Lambda_2,\theta_2).$ Finally, notice that $\mf f$ is the ERA Lie algebra associated to a discrete graph $\Delta$ with $\theta_\Delta:V(\Delta)\to \{0\}\subset\F_2,$ and hence the disjoint union $\Delta\sqcup\Lambda_1\sqcup\Lambda_2$ is a defining graph of $\mf m,$ where the signature is defined in the obvious way.
		
		Suppose now $\Gamma$ is connected. If $\theta\equiv 0,$ then $\mf e(\Gamma)$ is the RAAG Lie algebra $\mf a(\Gamma),$ which is Bloch-Kato when $\Gamma$ is a Droms graph. 
		On the other hand, if $\theta\not\equiv 0,$ then there exists a central vertex $v$ of $\Gamma$ with $\theta(v)=1,$ and hence $\mf e(\Gamma,\theta)\simeq\mf e(\Gamma_v,\theta')\times \mf k/\mf k\2,$ where $\theta'$ is the restriction of $\theta$ to the induced subgraph $\Gamma_v$ spanned by the vertices $\neq v.$ 
		If $\mf m$ is a standard subalgebra of $\mf e(\Gamma,\theta)\simeq \mf c(\Gamma_v,\theta')\times \mf k/\mf k\2,$ then we argue as in Lemma \ref{lem:directbk}, and prove that there exists a standard subalgebra $\mf h$ of $\mf g:=\mf e(\Gamma_v,\theta')$ such that $\mf m$ is isomorphic to either $\mf h$ or $\mf h\times \mf k/\mf k\2.$ By induction, $\mf h$ is an ERA Lie algebra, proving that the same is true for $\mf m.$
	\end{proof}
	In particular, this applies to RACG Lie algebras.
	
	\begin{cor}\label{prop:dromsRACG}
		Let $\mf c(\Gamma)$ be the RACG Lie algebra over $k$ associated to a simple graph $\Gamma.$ Then, the following are equivalent:
		\begin{enumerate}
			\item 	$\Gamma$ is a Droms graph.
			\item $\mf c(\Gamma)$ is a Bloch-Kato restricted Lie algebra.
			\item All the standard restricted subalgebras of $\mf c(\Gamma)$ are ERA Lie algebras.
			\item All the ordinary standard subalgebras of $\mf c(\Gamma)$ are RAAG Lie algebras \cite{sb}.
		\end{enumerate}
	\end{cor}

	\begin{cor}
		Let $V$ be a $k$-vector space with basis $(v_i)_{1\leq i\leq n},$ and let $\lambda:T(V)\to \Lambda(V)$ and $\sigma:T(V)\to S(V)$ be the natural projections. Let $E\subseteq\set{v_i\otimes v_j}{1\leq i,j\leq n}.$ Then,
		
		\begin{center}
			$\Lambda(V)/\lambda(E)$ is universally Koszul iff $S(V)/\sigma(E)$ is universally Koszul.
		\end{center}
		\begin{proof}
			Let $\Gamma=(V,E'),$ where $E'=\set{\{v_i,v_j\}}{v_i\otimes v_j\in E }.$ Then, the RACG $\mf c(\Gamma)$ is a Bloch-Kato restricted Lie algebra if and only if the RAAG $\li_\Gamma$ is a Bloch-Kato (ordinary) Lie algebra. In other words, the restricted cohomology of $\mf c(\Gamma)$ is universally Koszul precisely when the Chevalley-Eilenberg cohomology of $\li_\Gamma$ is universally Koszul. 
		\end{proof}
	\end{cor}


\end{document}